\documentclass[reqno]{amsart}
\usepackage{amsmath}
\usepackage{paralist}
\usepackage{amssymb}
\usepackage{amsthm}
\usepackage{graphicx}
\usepackage{tikz}
\usepackage{float}
\usepackage[colorlinks=true]{hyperref}
\hypersetup{urlcolor=blue, citecolor=red}
\usepackage{hyperref}
\newtheorem{theorem}{Theorem}[section]

\newtheorem{lemma}[theorem]{Lemma}
\newtheorem{proposition}[theorem]{Proposition}
\theoremstyle{definition}

\newtheorem{remark}{Remark}[section]

\numberwithin{equation}{section}

\begin{document}
\title[Normalized solutions to Kirchhoff equation]
{Normalized solutions to Kirchhoff equation with the Sobolev critical exponent in high dimensional spaces}
\author[Ruikang Lu, Qilin Xie and Jianshe Yu]{}

\maketitle
\medskip
{\footnotesize
\medskip
\centerline{\scshape Ruikang Lu, Qilin Xie\footnote{E-mail address: \ xieql@gdut.edu.cn}}
\medskip
{\footnotesize
\centerline{School of Mathematics and Statistics,}
\centerline{Guangdong University of Technology, Guangzhou, Guangdong 510520, China} }
\medskip
\centerline{and}
\medskip
\centerline{\scshape Jianshe Yu\footnote{E-mail address: \ jsyu@gzhu.edu.cn.}\footnote{This work was supported by National Natural Science Foundation of China (No: 12331017, 12171110)}}

{\footnotesize
\centerline{Guangzhou Center for Applied Mathematics,}
\centerline{Guangzhou University, Guangzhou, Guangdong 510006, China} }
\maketitle


\begin{abstract}
	The following well-known Kirchhoff equation with the Sobolev critical exponent has been extensively studied,
	\begin{equation*}
		-\Big(a+b\int_{\mathbb R^N} | \nabla u|^2dx\Big) \Delta u+\lambda u=\mu |u|^{q-2}u+|u|^{2^*-2}u \ \ {\rm in}\ \ \mathbb{R}^N, \ \ N\geq4,	
        \end{equation*}
having prescribed mass $\int_{\mathbb R^N}|u|^2dx=c$,
where $a$, $c$ are two positive constants, $b,\mu$ are two parameters, $\lambda$ appears as a real Lagrange multiplier and $2<q<2^*$,  $2^*$ is the Sobolev critical exponent. Firstly, for the special case $\mu=0$ and $N\geq4$, the above equation reduces to a pure critical Kirchhoff equation, we obtain a complete conclusion including the existence, nonexistence and multiplicity of the normalized solutions by the variational methods. Secondly, when $\mu>0$, $N\geq5$ and $2<q<2+\frac{4}{N}$, we investigate the existence of the positive normalize solution under suitable assumptions on  parameter $b$ and mass $c$. To the best of our knowledge, it is the first time to consider the above case, which is a more complicated case not only the difficulties on checking the Palais-Smale condition, but also the constraint functional requesting the intricate concave-convex structure. Lastly, when $\mu>0$ and $N=4$, we obtain a local minimizer solution and a mountain pass solution under explicit conditions on $b$ and $c$. It is worth noting that the second solution is obtained by introducing a new functional to establish a threshold for the mountain pass level, which is the key step for the fulfillment of the Palais-Smale condition.  This paper provides a refinement and extension of the results of the normalized solutions for Kirchhoff type problem in high-dimensional spaces.
\end{abstract}

\section{Introduction}
\setcounter{section}{1}
\setcounter{equation}{0}

 As is well-known, the Kirchhoff type equation appears as models of some physical phenomena. For example, it is related to the stationary analogue of the equation,
\begin{equation*}
	\rho\frac{\partial ^2 u}{\partial t^2}-(\frac{P_0}{h}+\frac{E}{2L}\int_{0}^{L}|\frac{\partial u}{\partial x}|^2dx)\frac{\partial ^2 u}{\partial x^2}=0,
\end{equation*}
where $u$ is the lateral displacement at $x$ and $t$, $E$ is the Young modulus, $\rho$ is the mass density, $h$ is the cross-section area, $L$ is the length, $P_0$ is the initial axial tension. For more background, refer to \cite{al2014bending, kirchhoffvorlesungen}. Compared with the Schr\"odinger equation, the Kirchhoff equation contains a nonlocal term $(\int_{\mathbb R^N}|\nabla u|^2dx) \Delta u$, which makes it no longer a pointwise identity. Thus, the nonlocal problem causes some mathematical difficulties and makes the search for solutions to the Kirchhoff equation more challenging and interesting. However, after Lions \cite{lions1978some} introduced an abstract functional analysis framework for the equation, $$u_{tt}-\Big(a+b\int_{\Omega} | \nabla u|^2dx\Big) \Delta u=f(u)\ \ {\rm in}\ \mathbb R^+\times\Omega,$$the Kirchhoff equation has attracted much attention from mathematicians. For the relevant results, see \cite{chipot1997some,correa2005class}.

A special form of the following Kirchhoff type equation,
\begin{equation}\label{p1.3}
    -\Big(a+b\int_{\mathbb R^N} | \nabla u|^2dx\Big) \Delta u+\lambda u=f(u)\ \ {\rm in}\ \ \mathbb{R}^N,
\end{equation}
has been widely studied in the past several decades, where $a,b$ are two positive constants, $\lambda$ is a fixed constant or  an unknown parameter, $N\geq 1$ and $f\in C(\mathbb R,\mathbb R)$. On the one hand, the equation \eqref{p1.3} with a fixed constant $\lambda$ is called fixed frequency problem. In this case, one can obtain the solutions to equation \eqref{p1.3} by finding the critical points of the related energy functional without constraints. Many results have been established in this case, and the reader may refer to \cite{chen2011nehari,figueiredo2014existence,he2012existence,xie2016bound,Xie2023existence} and the references therein for related work.

On the other hand, inspired by some physical phenomena, physicists are more interested in solutions to equation \eqref{p1.3} with a prescribed mass, which are commonly referred to as normalized solutions. In this case, $\lambda$ is an unknown parameter and appears as a Lagrange multiplier with respect to the constraint,
$$\bar{S}_c:=\{u\in H^1(\mathbb R^N): \int_{\mathbb R^N}|u|^2dx=c\} \ \ {\rm\  or }
\ \ S_c:=\{u\in H_{rad}^1(\mathbb R^N): \int_{\mathbb R^N}|u|^2dx=c\},$$
 where $c$ is a fixed constant. To obtain the normalized solutions to equation \eqref{p1.3}, we can find the critical points of the energy functional on the constraint
$\bar{S}_c$. Now, let us review the existing work on normalized solutions to equation \eqref{p1.3}.

When $a=1$, $b=0$ and $f(u)=\mu |u|^{q-2}u+|u|^{p-2}u$, equation \eqref{p1.3} reduces to the Schr\"odinger equation,
\begin{equation}\label{p1.4}
		-\Delta u+\lambda u=\mu |u|^{q-2}u+|u|^{p-2}u\ \  {\rm in}\ \ \mathbb{R}^N,
\end{equation}
where $u$ satisfies $\int_{\mathbb R^N}|u|^2dx=c$.  Recently, there are many important works on normalized solutions to this Schr\"odinger equation. For $N\geq3$, Soave \cite{soave2020normalized} first studied equation \eqref{p1.4} for $2<q<2+\frac{4}{N}$ and $2<q\leq2+\frac{4}{N}\leq p<2^*$, where $2+\frac{4}{N}$ is the $L^2$-critical exponent for the Schr\"odinger type equation, which is always strictly smaller than the Sobolev critical exponent  $2^*$ . The author proved the existence and asymptotic properties of normalized ground state solutions, as well as stability and instability results. In the Sobolev critical case \cite{soave2020sobolev}, the existence and asymptotic properties of normalized ground state solutions were established for $2<q<2+\frac{4}{N}$, $q=2+\frac{4}{N}$, $2+\frac{4}{N}<q<2^*$, respectively. It is worth noting that Jeanjean et al. \cite{jeanjean2022multiple} solved an open solution raised by Soave \cite{soave2020sobolev} and obtained a second normalized solution of equation \eqref{p1.4} with $N\geq4$ for $2<q<2+\frac{4}{N}$. Moreover, Wei and Wu \cite{wei2022normalized} further extended the results and obtained a second normalized solution of equation \eqref{p1.4} with $N=3$ for $2<q<2+\frac{4}{N}$.

When $a,b>0$ and $f(u)=|u|^{q-2}u$, equation \eqref{p1.3} becomes the following equation,
\begin{equation}\label{p1.5}
    -\Big(a+b\int_{\mathbb R^N} | \nabla u|^2dx\Big) \Delta u+\lambda u=|u|^{q-2}u \ \ {\rm in}\ \ \mathbb{R}^N,
\end{equation}
where $u$ satisfies $\int_{\mathbb R^N}|u|^2dx=c$. There are many results of equation \eqref{p1.5} in the past several decades. Ye \cite{ye2015sharp} pointed out that $2+\frac{8}{N}$ is the $L^2$-critical exponent for this equation, which plays a dominant role in the study of normalized solutions to equation \eqref{p1.5}. This exponent determines whether the constrained functional $I$ remains bounded from below on $\bar S_c$. By employing scaling techniques and the concentration-compactness principle, Ye \cite{ye2015existence,ye2015sharp,ye2016mass} proved the existence, nonexistence and concentration behavior of normalized solutions in three subcases $2<q<2+\frac{8}{N}$, $q=2+\frac{8}{N}$ and $2+\frac{8}{N}<q<2^*$. Additionally, Zeng and Zhang \cite{zeng2017existence} also established the existence and uniqueness of normalized solutions in this case.

If $f(u)=\mu |u|^{q-2}u+|u|^{p-2}u$, equation \eqref{p1.3} becomes the following equation,
\begin{equation}\label{p1.6}
    -\Big(a+b\int_{\mathbb R^N} | \nabla u|^2dx\Big) \Delta u+\lambda u=\mu |u|^{q-2}u+|u|^{p-2}u \ \ {\rm in}\ \ \mathbb{R}^N,
\end{equation}
where $u$ satisfies $\int_{\mathbb R^N}|u|^2dx=c$. For $\mu>0$ and $N=3$, Li et al. \cite{li2022normalized} considered the existence and asymptotic properties of normalized solutions. In the case $2<q<\frac{10}{3}$ and $\frac{14}{3}<p<6$, they obtained a local minimum solution and a mountain pass solution, respectively, by using the Ekeland's variational principle and constructing a special Palais-Smale sequence. Besides, in the Sobolev case, Li
 et al. \cite{li2022normalized} only obtained a local minimum solution to equation \eqref{p1.6} for $2<q<\frac{10}{3}$. However, it is worth to point that the energy functional also has a convex-concave structure in this case, which naturally guides us to consider the existence of the second solution. Recently, using different technical approaches, Feng et al. \cite{feng2023normalized} and Chen et al. \cite{chen2024normalized} obtained the second solution, which is a mountain pass solution. In \cite{hu2023normalized}, Hu and Mao proved the existence and nonexistence of normalized solutions by using constraint minimization and concentration compactness principle, which partially extend the results of \cite{li2022normalized} and \cite{soave2020normalized}. The case that $\mu<0$ also has been considered by the authors in \cite{carriao2022normalized}.  For more results,  the reader may refer to \cite{xu2024existence,xu2024multiplicity}.

As we know, the $L^2$-critical exponent $2+\frac{8}{N}$ and the Sobolev critical exponent $2^*=\frac{2N}{N-2}$ play a special role on the study of the Kirchhoff equation. The former serves as the threshold for the existence of normalized solutions, while the latter determines whether the Kirchhoff equation has a variational framework. It is easy to see that $2+\frac{8}{N}<\frac{2N}{N-2}$ for $N\leq3$, which explains why the previously mentioned works on the normalized solutions of the Kirchhoff equation focused on dimensions $N\leq3$. Therefore, there are few results on the normalized solutions of equation \eqref{p1.3} for $N\geq4$. When $1\leq N\leq 4$ and $f(u)=|u|^{p-2}u$ with $2<p<2^*$, Guo et al. \cite{guo2018blow} showed the existence and blow-up behavior of normalized solutions to equation \eqref{p1.3}
with trapping potential or constant vanishing potential. When $2^*=4$ for $N=4$, and  $f(u)=|u|^{p-2}u$ with $2<p\leq4$, Li et al. \cite{li2019existence} proved the existence and nonexistence of normalized minimizer and showed the delicate difference brought by the different potential cases.

In the present paper, we investigate the following Kirchhoff equation with a Sobolev critical exponent in high dimensional spaces,
\begin{equation}\label{p1.1}
		-\Big(a+b\int_{\mathbb R^N} | \nabla u|^2dx\Big) \Delta u+\lambda u=\mu |u|^{q-2}u+|u|^{2^*-2}u \ \ {\rm in}\ \ \mathbb{R}^N,	\ \ \  N\geq4,
\end{equation}
having prescribed mass $\int_{\mathbb R^N}|u|^2dx=c,$
where $a$, $c$ are two positive constants, $b,\mu$ are two parameters, $\lambda$ appears as a real Lagrange multiplier and $2<q<2^*$. From a variational point of view, the normalized solutions to equation \eqref{p1.1} can be obtained by finding the critical points of the following  energy functional $I_{b,\mu}:H^1(\mathbb{R}^N)\to \mathbb{R}$ defined by
\begin{equation}\label{I}
    I_{b,\mu}(u)=\frac{a}{2}|\nabla u|^2_2+\frac{b}{4}|\nabla u|^4_2-\frac{\mu}{q}|u|^q_q-\frac{1}{2^*}|u|^{2^*}_{2^*},
\end{equation}
on the constraint $\bar{S}_c$ or $S_c.$

Firstly, we consider equation \eqref{p1.1} with $N=4$,  it is great interest to see that $2+\frac{8}{N}=\frac{2N}{N-2}=4$, which is referred to as the mass-energy doubly critical phenomenon, both the mass critical exponent and the Sobolev critical exponent. This phenomenon leads to new difficulties  in the study of equation \eqref{p1.1}. Specifically, the competition between the nonlocal term and the Sobolev critical term in the functional \eqref{I} makes the analysis of its geometric structure and the proof of boundedness for Palais-Smale sequences more challenging. Moreover, if $q>2+\frac{8}{N}=4$, then equation \eqref{p1.1} will no longer have a variational framework. Therefore, $2+\frac{4}{N}=3$ is often considered the critical threshold instead of $2+\frac{8}{N}=4$ in the study of equation \eqref{p1.1}. Recalling the best Sobolev constant defined in \eqref{S}, we obtain the relationship between the nonlocal term and the Sobolev term of functional \eqref{I} as follows,
$$S^{-2}|\nabla u|_2^4\geq |u|_4^4,$$ 
which motivates us to establish  a connection between $b$ and $S^{-2}$ to investigate the geometric structure of functional \eqref{I}. For $0<b<S^{-2}$ and $\mu>0$, it follows from \eqref{I}, \eqref{S} and \eqref{G-N} that for all $u\in \bar{S}_c$, \begin{equation*}
            I_{b,\mu}(u)\geq\frac{a}{2}|\nabla u|_2^2-\frac{1-bS^2}{4S^2}|\nabla u|_2^4-\frac{\mu}{q}C_q^qc^{\frac{4-q}{2}}|\nabla u|_2^{2q-4}.
    \end{equation*} 
When $2<q<3$, $I_{b,\mu}(u)$ has  a convex-concave structure, suggesting that equation \eqref{p1.1} may have two normalized solutions: one is a local minimizer solution, and the other one is a mountain pass solution. Recently, Kong and Chen \cite{kong2023normalized} noted this case and prove the existence of the first solution under suitable conditions on $\mu$ and $c$, but the second solution remains unproven. Subsequently, using some energy estimations, Fang et al. \cite{fang2024normalized} obtained a mountain pass solution for $q=3$ and $\mu$ small enough. Additionally, Zhang et al. \cite{zhang2022normalized} considered equation \eqref{p1.1} with $\mu=1$ and $b>S^{-2}$, investigating the existence and nonexistence of normalized solutions in three subcases $2<q<3$, $q=3$ and $3<q<4$.

It is worth mentioning that there is little work on normalized solutions to equation \eqref{p1.1} for $N\geq5$. This is primarily due to the fact that $2^*<4$,  which makes the nonlocal term have a stronger impact on the energy functional than the Sobolev critical term. Consequently, some technical compactness arguments developed for $N\leq4$ are not applicable to $N\geq5$. In the case $N \geq5$,  the authors in \cite{zhang2022normalized} proved the existence and nonexistence of normalized solutions to equation \eqref{p1.1} with $\mu=1$. More precisely, their Theorem 1.5 obtained a normalized solution of equation \eqref{p1.1} as to the case of  $q$ in different specified intervals $(2,2^*)$  if  $a$, $b$ and $S$ satisfy the following crucial condition, \
\begin{equation}\tag{$b^*$}\label{b*}
(\frac{2a}{4-2^*})^{\frac{4-2^*}{2}}(\frac{2b}{2^*-2})^{\frac{2^*-2}{2}}>
\frac{1}{S^{\frac{2^*}{2}}}.
\end{equation}
Indeed, the condition \eqref{b*} equivalents to
\begin{equation}\label{b_0}
    b>b_0:=\frac{2}{N-2}
			\left(\frac{N-4}{a(N-2)}\right)^{\frac{N-4}{2}}S^{-\frac{N}{2}}.
\end{equation}
This vital constant $b_0$, which is given in \cite{xie2022study},  plays an significant role in structure of the energy functional and the limit equation. Based on the analysis of the preceding work, an intriguing question arises: For $b<b_0$, does equation \eqref{p1.1} have a normalized solution for $q\in(2,2^*)$,  $N\geq5$ and $\mu>0$ ? This question is one of the main aim of this paper. Before addressing it, we first study the existence, nonexistence and multiplicity of normalized solutions to equation \eqref{p1.1} with $\mu=0$. Subsequently, inspired by \cite{feng2023normalized,li2022normalized, Qizou2022, soave2020sobolev,xulin2024}, we establish partial results for the aforementioned question. Moreover, we prove the nonexistence results of equation \eqref{p1.1} with $\mu<0$. In the case $N=4$, this paper also address the open question raised in \cite{kong2023normalized}: Is there the second solution (mountain pass solution) to equation \eqref{p1.1} for $2<q<3$, $\mu>0$ and $0<b<S^{-2}$? Furthermore, we will use a different approach from that in \cite{kong2023normalized} to prove the existence of the first normalized solution (local minimizer).
\section{Main results}
\setcounter{section}{2}
\setcounter{equation}{0}

Before stating the main theorems, we first introduce some preliminary results. For $N\geq3$, the best Sobolev embedding constant $S$ is given by
\begin{equation}\label{S}
    S=\inf_{u\in D^{1,2}(\mathbb R^N)\setminus \{0\}}\frac{|\nabla u|_2^2}{|u|_{2^*}^2},
\end{equation}
and it is achieved by
\begin{equation}\label{U}
	U_{\varepsilon, y}(x)=\left(\frac{\sqrt{N(N-2)\varepsilon}}
	{\varepsilon+|x-y|^2}\right)^{\frac{N-2}{2}},\ \ \varepsilon>0,\ \  x\in\mathbb{R}^N,
\end{equation}
which is the unique positive solution of \eqref{S*},
\begin{equation}\tag{$S^*$}\label{S*}
	\begin{cases}
		-\Delta u=|u|^{2^*-2}u&{\rm in}\ \mathbb{R}^N,\\
		u\in D^{1,2}(\mathbb R^N).&
	\end{cases}
\end{equation}
Moreover, $S^{N/2}=|U|_{2^*}^{2^*}=|\nabla U|_2^{2}$.  In order to obtain the nontrivial solutions, we introduce the Pohozaev mainifold
$$\mathcal{P}_{c,b,\mu}:=\{u\in S_c: P_{b,\mu}(u)=0\},$$where  $$P_{b,\mu}(u):=a|\nabla u|_2^2+b|\nabla u|_2^4-\mu\delta_q|u|_q^q-|u|_{{2^*}}^{{2^*}}$$
with $\delta_q:=\frac{N(q-2)}{2q}$. By the Pohozaev inequality, it is easy to conclude that all critical points of the functional $I_{b,\mu}|_{S_c}$ belongs to $\mathcal{P}_{c,b,\mu}$. Thus, if $\inf_{u\in\mathcal{P}_{c,b,\mu}}I_{b,\mu}(u)$ is achieved by the normalized solution for equation \eqref{p1.1}, then it is a ground state radial solution of equation \eqref{p1.1} on $S_c$. For any $u\in S_c$ and $s\in\mathbb R$, the function $s\ast u$ defined by
$$(s\ast u)(x):=e^{\frac{N}{2}s}u(e^sx)\ \  {\rm for\ \  a.e. }\ x\in\mathbb R^N,$$
then $s\ast u\in S_c$. We also know that the map $(s,u)\in \mathbb R \times H^1(\mathbb R^N)\mapsto (s*u)\in H^1(\mathbb R^N)$ is continuous by Lemma 3.5 in \cite{bartsch2019multiple}. Moreover, we define the fiber map
$$\Psi_u^{b,\mu}(s):=I_{b,\mu}(s\ast u)=\frac{ae^{2s}}{2}|\nabla u|_2^2+\frac{be^{4s}}{4}|\nabla u|_2^4-\mu\frac{e^{q\delta_q s}}{q}|u|_{q}^{q}-\frac{e^{2^*s}}{2^*}|u|_{2^*}^{2^*}.$$
By some direct calculations, we have
$$(\Psi_u^{b,\mu})'(s)=ae^{2s}|\nabla u|_2^2+be^{4s}|\nabla u|_2^4 -\mu \delta_qe^{q\delta_q s}|u|_{q}^{q}-e^{2^*s}|u|_{2^*}^{2^*}=P_{b,\mu}(s\ast u),$$ and the Pohozaev mainifold can be rewrite as $$\mathcal{P}_{c,b,\mu}=\{u\in S_c: (\Psi_u^{b,\mu})'(0)=0\}.$$
Obviously, $s$ is a critical point of $\Psi_u^{b,\mu}(s)$ if and only if $s*u\in\mathcal{P}_{c,b,\mu}$, and $u\in\mathcal{P}_{c,b,\mu}$ if and only if 0 is a critical point of $\Psi_u^{b,\mu}(s)$.
To determine the exact location and types of some critical points for $I_{b,\mu}|_{S_c}$, we know that $\mathcal{P}_{c,b,\mu}$ can be split into the disjoint union $\mathcal{P}_{c,b,\mu}=\mathcal{P}^+_{c,b,\mu}\cup\mathcal{P}^0_{c,b,\mu}\cup\mathcal{P}^-_{c,b,\mu}$, where
$$\mathcal{P}^+_{c,b,\mu}:=\{u\in \mathcal{P}_{c,b,\mu}: ((\Psi_u^{b,\mu})''(0)>0 \},$$ $$\mathcal{P}^-_{c,b,\mu}:=\{u\in \mathcal{P}_{c,b,\mu}: ((\Psi_u^{b,\mu})''(0)<0 \},$$
$$\mathcal{P}^0_{c,b,\mu}:=\{u\in \mathcal{P}_{c,b,\mu}: ((\Psi_u^{b,\mu})''(0)=0 \},$$ for
$(\Psi_u^{b,\mu})''(0)=2a|\nabla u|_2^2+4b|\nabla u|_2^4-\mu q\delta^2_q|u|_{q}^{q}-2^*|u|_{2^*}^{2^*}.$ We recall the Gagliardo-Nirenberg inequality that exists an optimal constant $C_{N,q}$ depending on $N$ and $q$ such that
\begin{equation}\label{G-N}
	|u|_q\leq C_{N,q}|\nabla u|_2^{\delta_{q}}|u|_2^{1-\delta_q},\ \ u\in H^1(\mathbb{R}^N),
\end{equation}where $C_q:=C_{N,q}$ simplifies the notation.

Next, we simplify some notations. When $\mu=0$, we will replace $I_{b,\mu}$, $P_{b,\mu}$, $\Psi_u^{b,\mu}$, $\mathcal{P}_{c,b,\mu}$, $\mathcal{P}^{\pm}_{c,b,\mu}$ and $\mathcal{P}^0_{c,b,\mu}$ with $I_{0}$, $P_{0}$, $\Psi_u^{0}$, $\mathcal{P}_0$, $\mathcal{P}^{\pm}_0$ and $\mathcal{P}^0_0$, respectively. When $\mu\neq0$, we will replace $I_{b,\mu}$, $P_{b,\mu}$, $\Psi_u^{b,\mu}$, $\mathcal{P}_{c,b,\mu}$, $\mathcal{P}^{\pm}_{c,b,\mu}$ and $\mathcal{P}^0_{c,b,\mu}$ with $I$, $P$, $\Psi_u$, $\mathcal{P}$, $\mathcal{P}^{\pm}$ and $\mathcal{P}^0$, respectively. In the case $N\geq 5$, before addressing the perturbation case, it is necessary to study the existence, nonexistence and multiplicity of normalized solutions to equation \eqref{p1.1} with $\mu=0$.

\begin{theorem}\label{th1.1}
    Assume that $N\geq5$ and $\mu=0$, equation \eqref{p1.1} has the following results.
        \begin{itemize}
        \item [$(i)$] If $0<b<b_0$, equation \eqref{p1.1} exists two solutions satisfying $\varphi_{\pm}\in\mathcal{P}^{\pm}_{0}$ such that
		\begin{equation}\label{th1eq1}
			\inf_{u\in\mathcal{P}^-_{0}}I_0(u)=I_0(\varphi_{-})=c_{N,-}>0, \ \
			\inf_{u\in\mathcal{P}^+_{0}}I_0(u)=I_0(\varphi_{+})=c_{N,+},\ \
		\end{equation}
		and $c_{N,-}>c_{N,+}$.
		Moreover, $c_{N,+}\leq0$ for $b\in(0,b_1]$ and $c_{N,+}>0$ for $b\in(b_1,b_0)$, where $b_0$ is defined in \eqref{b_0} and  $b_1$ is defined by
		\begin{equation*}\label{b_0_1}
			b_1:=\frac{4}{N}\left(\frac{N-4}{aN}\right)^{\frac{N-4}{2}}
			S^{-\frac{N}{2}}.
		\end{equation*}
		In particular, $\varphi_-$ is a mountain pass solution in $S_c$ and
        $$\inf_{u\in S_c}I_0(u)=\begin{cases}
		I_0(\varphi_+), & {\rm if}\ \ b\in(0,b_1],\\
		0,&{\rm if}\ \ b\in(b_1,b_0).
	\end{cases}$$
        \item[$(ii)$] If $b<0$, there exists $\varphi_1\in\mathcal{P}_{0}$ such that
        \begin{equation}\label{th1eq2}
        \inf_{u\in\mathcal{P}_{0}}I_0(u)=\inf_{u\in\mathcal{P}^-_{0}}I_0(u)=I_0(\varphi_1)=c_N>0.
        \end{equation}
       In particular, $\varphi_1$ is a mountain pass solution in $S_c$.
       \item[$(iii)$] If $b>b_0$, then $I_0|_{S_c}$ has no critical points. Consequently, equation \eqref{p1.1} has no solutions in $S_c$.
       \item[$(iv)$] If  the equation \eqref{S*} is equipped with $n$ solutions in $L^2(\mathbb R^N)$, $n\in\mathbb{N} $, then  there exist $b_n>0$,($b_n\leq b_0$),  such that equation \eqref{p1.1} with $b\in(0,b_n)$ has $2n$ solutions $\varphi_{i,\pm}\in\mathcal{P}_{0}^\pm$, $i\in\{1,2,\cdots,n\}$ and there holds,
		\begin{align*}
                I_0(\varphi_{i,-})&>I_0(\varphi_{i,+}), \\
              I_0(\varphi_{1,\pm})&<I_0(\varphi_{2,\pm})<\cdots<I_0(\varphi_{n,\pm})<\frac{a^2}{N(N-4)b},
		\end{align*}
        where $\varphi_\pm=\varphi_{1,\pm}$.
    \end{itemize}
\end{theorem}
\begin{remark}
    More details about Theorem \ref{th1.1} are as following.
   \begin{itemize}
       \item [$(i)$] If $b=0$, equation \eqref{p1.1} reduces to the Schr\"odinger equation with Sobolev critical nonlinearity, which has normalized solutions in $S_c$ in \cite{soave2020sobolev}. By using the norm-preserving scaling, Theorem \ref{th1.1} extends the results of \cite{xie2022study}, which considered the general solution $u\in D^{1,2}(\mathbb R^N)$, to the normalized solution $u\in S_c$. Compared to the results of \cite{soave2020sobolev}, which considered the case $b=0$, the study of equation \eqref{p1.1} with $b>0$ is more difficult due to the additional impact of the nonlocal term $(\int_{\mathbb R^N}|\nabla u|^2dx) \Delta u$. Specifically, the energy functional $I(u)$ consists of four distinct terms, and their interplay significantly influences its structure, making it more difficult to determine the types of critical points on $S_c$. Moreover, the presence of the nonlocal term complicates the recovery of compactness for the Palais-Smale sequence. This is because the weak convergence $u_n\rightharpoonup u$ in $H^1(\mathbb R^N)$ does not guarantee the convergence $$
        |\nabla u_n|_2^2\int_{\mathbb R^N}\nabla u_n \nabla\varphi dx\to|\nabla u|_2^2\int_{\mathbb R^N}\nabla u\nabla\varphi dx \ \ {\rm for}\ \ {\rm all}\ \ \varphi \in \mathcal C_0^{\infty}(\mathbb R^N).$$
       \item [$(ii)$] In the proof of Theorem \ref{th1.1}, we constrain the functional $I_0$ on the Pohozaev mainifold $\mathcal{P}_{0}$ and know that $I_0|_{\mathcal{P}_{0}}$ is bounded from below. As we know, the structure of $I_0$ changes with the variation of $b$. It is obvious to check that $I_0$ has a concave-convex structure for $0<b<b_0$, and then we obtain two normalized solutions of equation \eqref{p1.1} by decomposing the Pohozaev mainifold and using a fiber map. Moreover, by decomposing the $S_c$ and constructing a minimax characterization, we find that one solution is a minimizer and the other is a mountain pass solution. Similarly, we obtain a mountain pass solution for $b<0$ and prove that there is no solution for $b> b_0$. Actually, the existence of two normalized solutions has been obtained in \cite{Qizou2022}, however, we not only obtain the existence of the normalized solutions, but also investigate some properties of those solutions.
   \end{itemize}

\end{remark}

Now we consider the perturbation case $\mu>0$.
\begin{theorem}\label{th2}
    Assume that $N\geq5$, $b_1<b<b_0$, $\mu>0$ and $2<q<2+\frac{4}{N}$.
    \begin{itemize}
        \item [$(i)$] There exists $\mu_2>0$ such that  if $\mu c^\frac{q(1-\delta_q)}{2}<\mu_2$, equation \eqref{p1.1} has a positive radial normalized solution with negative energy level, i.e., $m(b,\mu)<0$.
        \item [$(ii)$] If $u_{b,\mu}$ is a solution of equation \eqref{p1.1} obtained by $(i)$, then
    $$m(b,\mu)\to {0}\ \ and\ \ |\nabla u_{b,\mu}|_2^2\to 0 \ \ as\ \ \mu\to 0^+.$$
\end{itemize}
\end{theorem}
\begin{remark}
     There are some explanations about Theorem \ref{th2}.
    \begin{itemize}
        \item [$(i)$] Theorem \ref{th2} extends the results of \cite{feng2023normalized,kong2023normalized}, which studied equation \eqref{p1.1} in the cases $N=3,4$. Compared with the cases $N=3,4$, the nonlocal term $|\nabla u|_2^4$ always plays a dominant role in the four terms of $I(u)$ for the cases $N\geq5$, which makes the problem more difficult and interesting. For $2<q<2+\frac{4}{N}$, to recover the compactness of the Palais-Smale sequence, \cite{feng2023normalized} employed the continuity and monotonicity of the energy level for $N=3$, while \cite{kong2023normalized} added an additional constraint on $\mu,c$ for $N=4$. However, these methods are not applicable for $N\geq5$, we need to find a new method to overcome this difficulty.
        \item [$(ii)$]
       One of the most crucial steps in processing with the critical Sobolev problem is  overcoming the lack of compactness. The case $b<b_0$ is totally  different from the case $b>b_0$ in recovering the  compactness results.
        Recalling the Theorem 1.5 in Zhang et al. \cite{zhang2022normalized},  the authors recover the compactness results by using the essential condition $b>b_0$.  Actually, the limit equation (E.q. \eqref{p1.1} with $\mu=0$),  has no normalized solution, see Theorem \ref{th1.1} (iii) and it is a intrinsic condition in recovering the compactness results. More precisely, for $b>b_0$,  the function $h(t)$ is strictly monotonically increasing for any $t>0$, where
        \begin{equation}\label{h}
           h(t):=\frac{a}{2}t^2+\frac{b}{4}t^4-\frac{S^{-\frac{2^*}{2}}}{2^*}t^{2^*}.
        \end{equation}
        For any $u\in S_c$, by a simple calculation, one obtains that
        \begin{equation*}
		I(u)=\frac{a}{2}|\nabla u|_2^2+\frac{b}{4}|\nabla u|_2^4-\frac{1}{2^*}|u|_{2^*}^{2^*}-\frac{1}{q}|u|_q^q\ge h(|\nabla u|_2)-\frac{1}{q}C_q^qc^\frac{q(1-\delta_q)}{2}|\nabla u|_2^{q\delta_q}.
\end{equation*}
        Since $2<q<2+\frac{4}{N}$, it is obvious to check that the constraint functional $I|_{S_c}$ has a convex structure. Then, by using the minimization method and Ekeland's variational principle, one obtains a Palais-Smale sequence $\{u_n\}$ of the constraint functional $I|_{S_c}$ with $P(u_n)\to0$ as $n\to +\infty$. Last but no least, the Palais-Smale condition essential follows from  that the limit equation possesses no normalized solution.
        However, in the case $0<b<b_0$, it will be a more complicated case not only the difficulties on checking the Palais-Smale condition, but also the constraint functional requesting the intricate convex-concave-convex structure. In this paper, we establish the existence of a normalized solution for $b_1<b<b_0$. It is a interesting question that whether equation \eqref{p1.1} possesses more normalized solutions for $0<b<b_0$.
    \end{itemize}
\end{remark}
\begin{theorem}\label{th3}
Assume that $N\geq5$, $b>0$, $\mu<0$ and $2<q<2^*$.
\begin{itemize}
        \item [$(i)$] If $u$ is a critical point for $I|_{S_c}$ (not necessarily positive, or even real-valued), then the associated Lagrange multiplier $\lambda>0$ and $I(u)>c_{N,+}.$
        \item[$(ii)$]The equation
        \begin{equation}\label{th1.4eq}
            -\Big(a+b\int_{\mathbb R^N} | \nabla u|^2dx\Big) \Delta u=\lambda u+\mu |u|^{q-2}u+|u|^{2^*-2}u, \ \ u>0 \ \ {\rm in}\ \ \mathbb{R}^N
        \end{equation} has no solution $u\in H^1(\mathbb R^N)\cap u\in L^p(\mathbb R^N)$ for $\lambda>0$, where $p\in(0,\frac{N}{N-2}]$.
    \end{itemize}
\end{theorem}
\begin{remark}
    The proof of Theorem \ref{th3} is inspired by Theorem 1.2 in \cite{soave2020sobolev}, which considers the Schr\"odinger equation with $\mu<0$.
\end{remark}

 In the case $N=4$, we first give some results of equation \eqref{p1.1} with $\mu=0$.
\begin{theorem}\label{th5}
    Assume that $N=4$ and $\mu=0$,  it follows that
    $$\inf_{u\in S_c}I_0(u)=\begin{cases}
		0,&{\rm if}\ \ b\geq S^{-2},\\
	-\infty, &{\rm if}\ \ 0<b<S^{-2},
	\end{cases}$$
    where $\inf_{u\in S_c}I_0(u)$ is not achieved for $b\geq S^{-2}$.   Moreover,  for $0<b<S^{-2}$,    there holds that $$\inf_{u\in\mathcal{P}_{0}}I_0(u)=\inf_{u\in\mathcal{P}^-_{0}}I_0(u)=\Lambda
    =\frac{a^2S^2}{4(1-bS^2)},$$
which can not be achieved.
\end{theorem}

Before stating the results on $N=4$ of the perturbation case $\mu>0$, we define some numbers as follows
\begin{equation}\label{k0}
    k_0:=\frac{2aS^2(3-q)}{(1-bS^2)(4-q)},
\end{equation}
\begin{equation}\label{c0}
    c_0:=\left[\frac{2aS^2(3-q)}{(1-bS^2)(4-q)}\right]^2\left[\frac{q(1-bS^2)}{4\mu C_q^qS^2(3-q)}\right]^{\frac{2}{4-q}},
\end{equation}and
\begin{equation}\label{c1}
    c_1:= \left [ \frac{aqk_0^{3-q}}{2\mu C_q^q(4-q)(1-bS^2)} \right ]^\frac{2}{4-q}.
\end{equation}Besides, the set $A_{k_0}(c)$ is defined by \begin{equation}\label{Ak}
    A_{k_0}(c):=\{u\in \bar{S}_c:|\nabla u|_2^2<k_0\}.
\end{equation}
and the auxiliary functional  $J:H^1(\mathbb{R}^4)\to \mathbb{R}$ is defined by
\begin{equation}\label{Ju}
    J(u):=\frac{a}{2(1-bS^2)}|\nabla u|^2_2+\frac{b}{4(1-bS^2)}|\nabla u|^4_2-\frac{\mu}{q}|u|^q_q-\frac{1}{4}|u|^{4}_{4}.
\end{equation}

\begin{theorem}\label{th6}
    Assume that $N=4$, $0<b<S^{-2}$, $\mu>0$, $2<q<3$ and $c\in(0,c_0)$.
     \begin{itemize}
        \item [$(i)$] Equation \eqref{p1.1} has a positive local minimizer solution $\tilde{u}$ in $\bar{S}_c$ for some $\tilde{\lambda}_c>0$. Moreover, $\tilde{u}$ satisfies that \begin{equation*}
        \tilde{u}\in A_{k_0}(c), \ \ I(\tilde{u})<0.
    \end{equation*}
        \item [$(ii)$] The auxiliary functional $J(u)$ has a positive local minimizer $\bar{u}$ in $\bar{S}_c$ for some $\bar{\lambda}_c>0$.  Moreover, $\bar{u}$ satisfies that \begin{equation*}
        \bar{u}\in A_{k_0}(c), \ \ J(\bar{u})<0.
    \end{equation*}
\end{itemize}
\end{theorem}

\begin{remark}
    In fact, Theorem \ref{th6} $(i)$ is similar to Theorem 1.1 in \cite{kong2023normalized},  but the proof in \cite{kong2023normalized} relies on the decomposition of the Pohozaev manifold, while our method does not require such a decomposition. Additionally, for any $u\in S_c$, we have $J(u)\geq I(u)$ under the condition $0<b<S^{-2}$. Therefore, similar to the proof of Theorem \ref{th6} $(i)$, we can prove Theorem \ref{th6} $(ii)$, which plays an important role in the proof of next Theorem \ref{th7}.  In particular, for $b>S^{-2}$ and $2<q<3$, it is easy to check that $I(u)$ has a convex structure, which means that $I(u)$ is coercive. By the minimization arguments, \cite{zhang2022normalized} proved that equation \eqref{p1.1} has a global minimizer for all $c>0$.
\end{remark}

\begin{theorem}\label{th7}
     Assume that $N=4$, $0<b<S^{-2}$, $\mu>0$, $2<q<3$ and $c\in(0,\min\{c_0,c_1\})$. Then equation \eqref{p1.1} has a mountain pass normalized solution $\hat{u}$ in  $S_c$ for some $\hat{\lambda}_c>0$. Moreover, $\hat{u}$ satisfies that $$0<I(\hat{u})<J(\bar{u})+\Lambda.$$
\end{theorem}

\begin{remark}\label{rem1.5}
      For  the case $b=0$, the authors in \cite{chen2024another, jeanjean2022multiple,wei2022normalized}, to prove the existence of the mountain pass solution for $2<q<3$, they constructed different test functions to obtain the same threshold of mountain pass level as follows $$M(c)<m(c)+\frac{S^2}{4},$$ which is a crucial step in recovering the compactness of Palais-Smale sequences, where $M(c)$ is the mountain pass level and $m(c)$ is the ground state energy. Indeed, many properties of the functional are preserved under small perturbations. Therefore, when $0<b<S^{-2}$, we also establish a threshold of the mountain pass level to prove Theorem \ref{th7}.
\end{remark}

\begin{remark}\label{rem1.4}
    In the case $N=3$ and $2<q<\frac{10}{3}$, the authors in \cite{feng2023normalized} and \cite{chen2024normalized} established different thresholds of the mountain pass level to obtain the mountain pass solution. The former considered the superposition of the ground state solution of equation \eqref{p1.1} with $b=0$ and the Aubin-Talenti bubbles to construct a test function in $S_c$. After some delicate calculations, the following strict inequality was obtained $$M(c)<m(c)+\frac{(aS)^{\frac{3}{2}}}{3}.$$ The latter introduced an auxiliary functional and
proved the existence of its local minimizer. By combining this minimizer with the Aubin-Talenti bubbles, a different test functional was constructed, leading to the following strict inequality

$$M(c)<\bar{m}(c)+\frac{abS^3}{4}+\frac{b^3S^6}{24}+\frac{(b^2S^4+4aS)^\frac{3}{2}}{24},$$
where $\bar{m}(c)$ is the energy of minimizer. In the proof of Theorem \ref{th2}, inspired by  \cite{chen2024normalized}, we also introduce an auxiliary function $J(u)$ defined in \eqref{Ju} and prove its local minimizer. Then, we consider the superposition of this minimizer with Aubin-Talenti bubbles to construct a test function, ultimately obtaining the following threshold of mountain pass level  for $N=4$,
$$M(c)<\bar{m}(c)+\frac{a^2S^2}{4(1-bS^2)}.$$
In particular, when $a=1$ and $b=0$, the above inequality degenerates into the threshold of mountain pass level corresponding to the Schr\"odinger equation mentioned in Remark \ref{rem1.5}.
\end{remark}

This paper is organized as follows. In Section \ref{sec2}, we present the proof of the pure critical Kirchhoff equation, that is,  $\mu=0$, $N\geq5$ and $N=4$, and complete the proof of Theorem \ref{th1.1}, Theorem \ref{th3} and  Theorem \ref{th5}. In Section \ref{sec3}, we establish the compactness result of the Palais-Smale sequences for $I(u)$ and complete the proof of Theorem \ref{th2}.  Finally, we give the proof of Theorems \ref{th6} and \ref{th7} in Sections \ref{sec6}. We introduce some notations that will be used in this paper.
\begin{itemize}
  \item $H_{rad}^1(\mathbb R^N)=\{u\in H^1(\mathbb R^N): u(x)=u(|x|)\}$.
  \item $D^{1,2}(\mathbb R^N)$ is the completion of $C_0^{\infty}(\mathbb R^N)$ with respect to  $\left\|u \right\|_{D^{1,2}}=|\nabla u|_2$.
  \item $|\cdot|_q$ is the standard norm in $L^q(\mathbb R^N)$ for $q\geq1$.
  \item $o(1)$ means a quantity that tends to 0.
  \item $\rightarrow$ and $\rightharpoonup$ denote the strong convergence and weak convergence, respectively.
  \item For any $x\in\mathbb R^N$ and $r>0$, $B_r(x):=\{y\in \mathbb R^N:|y-x|<r\}$.
  \item $C$, $M$ are positive constants possibly different in different places.
\end{itemize}

\section{The proof of Theorems \ref{th1.1},  \ref{th3} and Theorem \ref{th5}}\label{sec2}
In this section, we mainly introduce the results of  the pure critical Kirchhoff equation, i.e., the case $\mu=0$, which will be used in the rest of the paper.
\begin{proof}[The proof of Theorem  \ref{th1.1}]
	For $N\geq5$,  we assume that
		$$E_0(u)=\frac{1}{2}|\nabla u|^2_2-\frac{1}{2^*}|u|^{2^*}_{2^*}$$
	and
		$\mathcal{P}_{c}:=\{u\in S_c: |\nabla u|_2^2=|u|_{2^*}^{2^*}\}.$
	By Proposition 2.2 in \cite{soave2020sobolev}, we know that $\inf_{u\in\mathcal{P}_{c}}E_0(u)=\frac{1}{N}S^{\frac{N}{2}}$ is achieved by $U_{\varepsilon_c,0}$, which
is defined in \eqref{U} for the unique choice of $\varepsilon_c>0$ such that $|U_{\varepsilon_c,0}|_2=c$.

	$(i)$ For $0<b<b_0$, we denote the two roots of $g(t)=0$ by $t_\pm$, $0<t_-<t_+$, where
	\begin{equation}\label{gt}
		g(t):=bS^{\frac{N}{2}}t^2-t^{\frac{4}{N-2}}+a.
	\end{equation}
	Moreover, $g'(t_-)<0$ and $g'(t_+)>0$.
	Let $s_\pm$ satisfy $e^{s_\pm}=t_\pm$, and set that
	
		$$\varphi_\pm(x):= (s_{\pm}\ast U_{\varepsilon_c,0})(x)=e^{\frac{N}{2}s_{\pm}}U_{\varepsilon_c,0}(e^{s_{\pm}}x)\in S_c.$$
By the above setting and the fact that $\pm g'(t_\pm)>0$, we can check that $\varphi_\pm\in \mathcal{P}^{\pm}_{0}$.
	Thanks to the assumption $0<b<b_0$, we define $\xi_\pm$ ($\xi_-<\xi_+$) as the roots of equation $f(t)=0$, where
	\begin{equation}\label{ft}
		f(t):=bt^2-S^{-\frac{N}{N-2}}t^{\frac{4}{N-2}}+a.
	\end{equation}
	Moreover, $f'(\xi_-)<0$,  $f'(\xi_+)>0$, and
    \begin{equation}\label{xi}
        \xi_\pm=S^{\frac{N}{4}}t_\pm=|\nabla\varphi_\pm|_2.
    \end{equation}
For any $u\in\mathcal{P}_{0}$, by \eqref{S}, we have
	\begin{equation*}
		a|\nabla u|_2^2+b|\nabla u|_2^4=|u|_{2^*}^{2^*} \leq S^{-\frac{2^*}{2}}|\nabla u|_{2}^{2^*},
	\end{equation*}
	which implies $\xi_-\leq |\nabla u|_2\leq \xi_+$. It follows that
	\begin{equation*}
		\begin{split}
			(\Psi_u^{0})''(0)&=2a|\nabla u|_2^2+4b|\nabla u|_2^4-2^*|u|_{2^*}^{2^*}-2^*(\Psi_u^{0})'(0) \\
			&=\frac{2(N-4)}{N-2}b|\nabla u|_2^2\left(|\nabla u|_2^2-\frac{2a}{(N-4)b}\right)\\
			&=\frac{2(N-4)}{N-2}b|\nabla u|_2^2\left(|\nabla u|_2^2-\eta^2\right),
		\end{split}
	\end{equation*}
	where $\eta:=(\frac{2a}{(N-4)b})^{1/2}$. Then we know that
	\begin{equation*}
	    \begin{split}
			&\mathcal{P}^-_{0}=\{u\in \mathcal{P}_{0}: \xi_-\leq |\nabla u|_2<\eta\},  \\& \mathcal{P}^+_{0}=\{u\in \mathcal{P}_{0}: \eta< |\nabla u|_2\leq \xi_+\},\\
			&\mathcal{P}^0_{0}=\{u\in \mathcal{P}_{0}: |\nabla u|_2=\eta\},
		\end{split}
	\end{equation*}
	For any $u\in\mathcal{P}_{0}$, then we have
	\begin{equation}\label{CN-}
		\begin{split}
			\inf_{u\in\mathcal{P}^-_{0}}I_0( u)&=\inf_{u\in\mathcal{P}^-_{0}}\left\{\frac{a}{2}|\nabla u|_2^2+\frac{b}{4}|\nabla u|_2^4-\frac{1}{2^*}|u|_{2^*}^{2^*}-\frac{1}{2^*}(\Psi_u^{0})'(0)\right\}\\
			&=\inf_{u\in\mathcal{P}^-_{0}}\left\{\frac{1}{N}a|\nabla u|_2^2-\frac{N-4}{4N}b |\nabla u|_2^4\right\}\\&=\frac{1}{N}a\xi_-^2-\frac{N-4}{4N}b \xi_-^4=c_{N,-}=I_0(\varphi_-).
		\end{split}
	\end{equation}
	Similarly, it follows that
	\begin{equation}\label{CN+}
		\begin{split}
			&\inf_{u\in\mathcal{P}^+_{0}}I_0( u)=\frac{1}{N}a\xi_+^2-\frac{N-4}{4N}b \xi_+^4=c_{N,+}=I_0(\varphi_+).
		\end{split}
	\end{equation}
	Moreover, by Theorem 1.2 in Xie and Zhou \cite{xie2022study}, we know that
	$c_{N,-}>0$, $c_{N,-}>c_{N,+}$ for any $0<b<b_0$, $c_{N,+}\leq0$ for $0<b\leq b_1$ and $c_{N,+}>0$ for $b_1<b<b_0$.

    We claim that
        $$\inf_{u\in A_{\xi_-}}I_0(u)=0\ \ {\rm and} \ \ \inf_{u\in A_{\xi_-}^c}I_0(u)=I_0(\varphi_+),$$
    where $A_{\xi_-}:=\{u\in S_c: |\nabla u|_2<\xi_-\}$ and $\xi_-$ is a constant defined by \eqref{xi}. Moreover, $$\inf_{u\in S_c}I_0(u)=\begin{cases}
		I_0(\varphi_+), & {\rm if}\ \ b\in(0,b_1],\\
		0,&{\rm if}\ \ b\in(b_1,b_0).
	\end{cases}$$
    Then, $\varphi_+$ is a global minimum solution in $S_c$ for $0<b\leq b_1$, and a local minimum solution in $S_c$ for $b_1<b<b_0$.

    In fact,  by \eqref{S} and direct computations, we obtain that
    $$I_0(u)\geq\frac{a}{2}|\nabla u|_2^2+\frac{b}{4}|\nabla u|_2^4-\frac{S^{-\frac{2^*}{2}}}{2^*}|\nabla u|_2^{2^*}=h(|\nabla u|_2)>0 \ \text{for\  any}\  u\in A_{\xi_-},$$
    where $h(t)$ is defined in \eqref{h}. Hence, $\inf_{u\in A_{\xi_-}}I_0(u)\geq0$. For any $s\in\mathbb R$, set that
    \begin{equation}\label{us}
        u_s(x):=(s* U_{\varepsilon_c,0})(x)=e^{\frac{N}{2}s}U_{\varepsilon_c,0}(e^{s}x)\in S_c.
    \end{equation}Since $|\nabla u_s|_2=e^sS^\frac{N}{4}<\xi_-$ for $s<-M$, one obtains that $u_s\in A_{\xi_-}$ and \begin{equation}\label{Ius}
        I_0(u_s)=e^{2s}\frac{a}{2}S^\frac{N}{2}+e^{4s}\frac{b}{4}S^N-e^{2^*s}\frac{1}{2^*}S^\frac{N}{2}\to0 \ \ {\rm as}\ \ s\to-\infty,
    \end{equation}which implies $\inf_{u\in A_{\xi_-}}I_0(u)=0$. Furthermore, since $I_0(u)\geq h(|\nabla u|_2)>0$  for any $u\in A_{\xi_-}$, $\inf_{u\in A_{\xi_-}}I_0(u)$ has no minimizer.

    On the one hand, for any $u\in \mathcal{P}_{0}^+$, we have $|\nabla u|_2>\eta>\xi_-$, which implies $ \mathcal{P}_{0}^+\subset A_{\xi_-}^c$.Then, it follows that $\inf_{u\in \mathcal{P}_{0}^+}I_0(u)\geq\inf_{u\in A_{\xi_-}^c}I_0(u)$. On the other hand, from Remark 1.5 in \cite{xie2022study}, we also know that
    \begin{equation}\label{h_xi_+}
        \inf_{u\in A_{\xi_-}^c}I_0(u)\geq\inf_{t\geq\xi_-}h(t)=h(\xi_+)=c_{N,+}=\inf_{u\in \mathcal{P}_{0}^+}I_0(u).
    \end{equation} Thus, $\inf_{u\in A_{\xi_-}^c}I_0(u)=I_0(\varphi_+)=c_{N,+}$ holds. Moreover, by the fact that $c_{N,+}\leq0$ for $0<b\leq b_1$ and $c_{N,+}>0$ for $b_1<b<b_0$, we have $$\inf_{u\in S_c}I_0(u)=\begin{cases}
		I_0(\varphi_+), & {\rm if}\ \ b\in(0,b_1],\\
		0,&{\rm if}\ \ b\in(b_1,b_0).
	\end{cases}$$
We also claim that $\varphi_-$ is a mountain pass solution in $S_c$ and $c_{N,-}$ is the  mountain pass level. In fact, by \eqref{us}, \eqref{Ius} and $c_{N,-}>0$, one can find a $ s_0<-M$ such that $u_{s_0}\in A_{\xi_-}$ and $I_0(u_{s_0})<c_{N,-}=I_0(\varphi_-)$. Combining with $c_{N,+}<c_{N,-}$ and $|\nabla\varphi_-|_2=\xi_-<\xi_+=|\nabla\varphi_+|_2$, we can give that $$\max\{I_0(u_{s_0}),I_0(\varphi_+)\}<I_0(\varphi_-)\leq\inf_{u\in\partial A_{\xi_-}}I_0(u).$$
    Thus, we can define a mountain pass level as following
    $$m_-=\inf_{\gamma\in\Gamma}\sup_{t\in[0,1]}I_0(\gamma(t)),$$
    where $\Gamma=\{\gamma(t)\in C([0,1])\cap S_c:\gamma(0)=u_{s_0} \ {\rm and} \ \gamma(1)=\varphi_+\}$. Then there holds $$m_-=c_{N,-}.$$ Since $|\nabla\gamma(0)|_2<\xi_-<|\nabla\gamma(1)|_2$ for any $\gamma\in \Gamma$, there exists $t_\gamma\in(0,1)$ such that $\gamma(t_\gamma)\in\partial A_{\xi_-}$. Thus, it follows that $$\sup_{t\in[0,1]}I_0(\gamma(t))\geq I_0(\gamma(t_\gamma))\geq\inf_{u\in\partial A_{\xi_-}}I_0(u)\geq I_0(\varphi_-),$$ which implies $m_-\geq c_{N,-}$. Setting that $r(t)=ts_++(1-t)s_0$ and $\gamma_1(t)=u_{r(t)}\in C([0,1])\cap S_c$ in \eqref{us}, where $$\gamma_1(0)=u_{r(0)}=s_0*U_{\varepsilon_c,0}=u_{s_0} \ \ {\rm and} \ \ \gamma_1(1)=u_{r(1)}=s_+*U_{\varepsilon_c,0}=\varphi_+,$$ then there holds $\gamma_1(t)\in \Gamma$ and $$m_-\leq\sup_{t\in[0,1]}I_0(\gamma_1(t))=\max_{t\in[0,1]}I_0(u_{r(t)})=I_0(\varphi_-),$$ which implies $m_-\leq c_{N,-}$. Therefore, $m_-=c_{N,-}$ holds.

    $(ii)$ For $b<0$, it is easy to check that the equation $g(t)=0$ has a positive root $t_1$. Moreover, we have $g'(t_1)<0$. Let $s_1$ satisfy $e^{s_1}=t_1$ and set that
        $$\varphi_1(x):= (s_1* U_{\varepsilon_c,0})(x)=e^{\frac{N}{2}s_1}U_{\varepsilon_c,0}(e^{s_1}x)\in S_c,$$
    then combining with $g'(t_1)<0$ implies $\varphi_1\in \mathcal{P}^-_{0}$. Setting that\begin{equation}\label{xi1}
        \xi_1=|\nabla\varphi_1|_2,
    \end{equation}
    then $\xi_1$ is a positive root of $f(t)=0$. Moreover, $f(t)>0$ for $t<\xi_1$ and $f(t)<0$ for $t>\xi_1$. It follows from $u\in \mathcal{P}_{0}$ and \eqref{S} that $$a|\nabla u|_2^2+b|\nabla u|_2^4=|u|_{2^*}^{2^*} \leq S^{-\frac{2^*}{2}}|\nabla u|_{2}^{2^*},$$which implies $|\nabla u|_2\geq\xi_1$. For any $u\in \mathcal{P}_{0}$, by $a>0$ and $b<0$, we have $$(\Psi_u^{0})''(0)=\frac{2(N-4)}{N-2}b|\nabla u|_2^4-\frac{4a}{(N-4)b}|\nabla u|_2^2<0,$$which implies $u\in \mathcal{P}^-_{0}$, thus $\mathcal{P}_{0}=\mathcal{P}^-_{0}$. At last, similar to \eqref{CN-}, we have
        $$\inf_{u\in\mathcal{P}_{0}}I_0( u)=\inf_{u\in\mathcal{P}^-_{0}}I_0( u)=\frac{1}{N}a\xi_1^2-\frac{N-4}{4N}b \xi_1^4=c_N=I_0(\varphi_1)>0.$$
 We claim that $\varphi_1$ is a mountain solution in $S_c$ and $c_N$ is the  mountain pass level. In fact, by \eqref{us}, one obtains that
     $$\lim_{s\to -\infty}|\nabla u_s|_2=0, \lim_{s\to -\infty}I_0(u_s)=0,\lim_{s\to +\infty}|\nabla u_s|_2=+\infty, \lim_{s\to +\infty}I_0(u_s)=-\infty.$$Then, we can take $s_2< -M$ such that $|\nabla u_{s_2}|<\xi_1$ and $I_0(u_{s_2})<c_N$, and take $s_3> M$ such that $|\nabla u_{s_3}|>\xi_1$ and $I_0(u_{s_3})<0$, where $\xi_1$ is a constant defined by \eqref{xi1}. Hence, it follows that
    $$\max\{I_0(u_{s_2}),I_0(u_{s_3})\}<I_0(\varphi_1)\leq\inf_{u\in\partial A_{\xi_1}}I_0(u).$$
    Thus, we can define a mountain pass level as following
    $$m_1=\inf_{\gamma\in\Gamma}\sup_{t\in[0,1]}I_0(\gamma(t)),$$
    where $\Gamma=\{\gamma(t)\in C([0,1])\cap S_c:\gamma(0)=u_{s_2} \ {\rm and} \ \gamma(1)=u_{s_3}\}$. Finally, it is easy to prove that $$c_N=m_1.$$

    $(iii)$ For $b>b_0$, it is easy to check that $f(t)>0$ by some direct computations.
Assume that $u$ is a critical point of $I_0(u)$ on $S_c$, we have $u\in \mathcal{P}_{0}$. Then, there holds
    $$0=a|\nabla u|_2^2+b|\nabla u|_2^4-|u|_{2^*}^{2^*}\geq a|\nabla u|_2^2+b|\nabla u|_2^4-S^{-\frac{2^*}{2}}|\nabla u|_2^{2^*}=|\nabla u|_2^2f(|\nabla u|_2)\geq0,$$
which implies $|\nabla u|_2=0$ and  $u\equiv0$ contradicts to $u\in S_{c}$.

    $(iv)$ For any $n\in\mathbb{N}$, let $\{u_i\}_{i=1}^n$ be $n$ solutions of equation \eqref{S*} satisfying \begin{equation}\label{u_inequality}
        |\nabla U|_2=|\nabla u_1|_2<|\nabla u_2|_2<\cdots<|\nabla u_n|_2.
    \end{equation}Set that
    $$v_i^\varepsilon(x):=\varepsilon^{\frac{N-2}{2}}u_i(\varepsilon x),\ \ 1\leq i\leq n,\ \ \varepsilon>0,$$ which are solutions of equation \eqref{S*} as well. Moreover, we can choose a unique $\varepsilon_i$ such that $|v_i^{\varepsilon_i}|_2=c$, i.e., $v_i^{\varepsilon_i}\in S_c$. Let $$b_n=\frac{2}{N-2}(\frac{N-4}{a(N-2)})^{\frac{N-4}{2}}|\nabla u_n|_2^{-2},$$ it is easy to cheak that the existence of two positive roots $t_{i,\pm}(t_{i,-}<t_{i,+})$ for equation $g_i(t)=0$ defined by $$g_i(t)=b|\nabla v_i^{\varepsilon_i}|_2^2t^2-t^{\frac{4}{N-2}}+a,\ \ t\geq 0,$$where $i\in\{1,2,\cdots,n\}$ and $b\in (0,b_n)$. Let $s_{i,\pm}$ satisfy $e^{s_{i,\pm}}=t_{i,\pm}$ and $$\varphi_{i,\pm}(x):=(s_{i,\pm}*v_i^{\varepsilon_i})(x)\in S_c,$$
    then we can check that $\varphi_{i,\pm}\in \mathcal{P}_{0}$. Set $\xi_{i,\pm}:=|\nabla\varphi_{i,\pm}|_2$ for $i\in\{1,2,\cdots,n\}$, then $\xi_{i,\pm}$ are the roots of equations $f_i(t)=0$ defined by $$f_i(t)=bt^2-(|\nabla v_i^{\varepsilon_i}|_2^{-1})^\frac{4}{N-2}+a,\ \ t\geq0,$$where $i\in\{1,2,\cdots,n\}$ and $b\in (0,b_n)$. With a similar argument of Lemma 3.3 in \cite{xie2022study}, one obtains $$\xi_{i,-}<\eta<\xi_{i,+}\ \ {\rm and} \ \ \xi_{i,+}^2-\eta^2>\eta^2-\xi_{i,-}^2,$$
    which implies $\varphi_{i,\pm}\in\mathcal{P}_{0}^\pm$ and $I_0(\varphi_{i,-})>I_0(\varphi_{i,+})$. By the above setting and \eqref{u_inequality}, we know that $|\nabla v_i^{\varepsilon_i}|_2<|\nabla v_{i+1}^{\varepsilon_{i+1}}|_2$ for $i\in \{1,2,\cdots,n-1\}$, which implies that $f_i(t)<f_{i+1}(t)$ and
    \begin{equation}\label{xi-inequality}
        0<\xi_{1,-}<\xi_{2,-}<\cdots<\xi_{n,-}<\eta<\xi_{n,+}<\cdots<\xi_{2,+}<\xi_{1,+}.
    \end{equation}
For any $u\in\mathcal{P}_{0}$, we have
    $$I_0(u)=I_0(u)-\frac{1}{2^*}(\Psi_u^{0})'(0)=\frac{1}{N}a|\nabla u|_2^2-\frac{N-4}{4N}b |\nabla u|_2^4:=\mathcal I_0(|\nabla u|_2^2),$$
    where $$\mathcal I_0(t)=\frac{1}{N}at-\frac{N-4}{4N}bt^2, \ \ t\geq0.$$ It is easy to check that $\mathcal I_0(t)$ is increasing in $(0,\eta^2)$ and decreasing in $(\eta^2,+\infty)$.
    Combining with $(i)$ and \eqref{xi-inequality}, one obtians $\varphi_\pm=\varphi_{1,\pm}$ and
    $$I_0(\varphi_{1,\pm})<I_0(\varphi_{2,\pm})<\cdots<I_0(\varphi_{n,\pm})<\frac{a^2}{N(N-4)b}.$$
The proof is complete.
\end{proof}

\begin{proof}[The proof of Theorem \ref{th3}]
    $(i)$ Let $u$ be a critical point of $I(u)$ on $S_c$. Then $u$ solves \eqref{p1.1} for some $\lambda\in\mathbb R$, testing \eqref{p1.1} by $u$ and interaing over $\mathbb R^N$, we have
    \begin{equation}\label{eq4.1}
    a|\nabla u|_2^2+b|\nabla u|_2^4+\lambda|u|_2^2=\mu|u|_q^q+|u|_{2^*}^{2^*}.
     \end{equation}
     Recalling that $P(u)=0$, i.e.,
     \begin{equation}\label{eq4.2}
         P(u)=a|\nabla u|_2^2+b|\nabla u|_2^4-\mu\delta_q|u|_q^q-|u|_{2^*}^{2^*}=0,
     \end{equation}
     then we infer that $\lambda|u|_2^2=\mu(1-\delta_q)|u|_q^q$ by \eqref{eq4.1} and \eqref{eq4.2}. Since $S_c\ni u \not\equiv 0$, $\mu<0$ and $0<\delta_q<1$, we conclude that $\lambda<0$. It follows from $P(u)=0$ and \eqref{S} that
    $$a|\nabla u|_2^2+b|\nabla u|_2^4=\mu\delta_q|u|_q^q+|u|_{2^*}^{2^*}\leq|u|_{2^*}^{2^*}\leq\frac{S^{-\frac{2^*}{2}}}{2^*}|\nabla u|_2^{2^*}.$$
    Thus, by \eqref{ft}, we have
    \begin{equation}\label{eq4.3}
        \xi_-<|\nabla u|_2<\xi_+.
    \end{equation}
    Moveover,
    \begin{align*}
        I(u)&= I(u)-\frac{1}{2^*}P(u)=\frac{1}{N}a|\nabla u|_2^2-\frac{N-4}{4N}b|\nabla u|_2^4-\frac{\mu}{q}(1-\frac{\delta_q q}{2^*})|u|_q^q\\
        & \geq\frac{1}{N}a|\nabla u|_2^2-\frac{N-4}{4N}b|\nabla u|_2^4\geq\frac{1}{N}a\xi_+^2-\frac{N-4}{4N}b\xi_+^4\\&=c_{N,+},
    \end{align*}
    where $0<\delta_q q<2^*$, $\mu<0$, \eqref{CN+} and \eqref{eq4.3} are used.

    $(ii)$ Assume that $u\in H^1(\mathbb R^N)\cap L^p(\mathbb R^N)$ is a solution of equation \eqref{th1.4eq}, where $p\in (0,\frac{N}{N-2}]$. By the Br\'ezis-kato argument \cite{brezis1979remarks}, we know that $u,|\Delta u|\in L^\infty(\mathbb R^N)$, and standard gradient estimates for the Poisson equation also imply that $|\nabla u|\in L^\infty(\mathbb R^N)$. It follows from $u\in L^2(\mathbb R^N)$ that
    \begin{equation}\label{ux}
        u(x)\to 0\ \ {\rm as}\ \ |x|\to\infty.
    \end{equation}
    Then, since $|\nabla u|\in L^\infty(\mathbb R^N)$, there exists $R_0>0$ such that $|\nabla u|_2^2\leq1$ for $|x|\geq R_0$. Combining with the fact that $\lambda<0$ from $(i)$, we have
    $$-\Delta u=\frac{1}{a+b|\nabla u|_2^2}(-\lambda u+\mu|u|^{q-2}u+|u|^{2^*-2}u)\geq \frac{-\lambda}{a+b}>0\ \ {\rm for} \ \ |x|>R_0.$$
    Thus, $u$ is superharmonic at infinity. Now we define
    $$m(r):=\min_{|x|=r} u(x).$$
    Since $u>0$ in $\mathbb R^N$, we know that $m(r)>0$. The Hadamard three spheres theorem in \cite{weinberger1967maximum} implies that $m(r)$ satisfies
    \begin{equation}\label{mr}
        m(r)\geq\frac{m(r_1)(r^{2-N}-r_2^{2-N})+m(r_2)(r_1^{2-N}-r^{2-N})}{r_1^{2-N}-r_2^{2-N}}\ \ {\rm for}\ \ R_0<r_1<r<r_2.
    \end{equation}By \eqref{ux}, we obtain that $m(r_2)\to0$ as $r_2\to \infty$. Let $r_2\to \infty$ in \eqref{mr}, then we know that the function $r\mapsto r^{N-2}m(r)$ is monotone non-decreasing for $r>R_0$. Therefore,
    $$m(r)\geq\frac{m(R_0)R_0^{N-2}}{r^{N-2}}\ \ {\rm for} \ \ r>R_0,$$
    which implies that
    $$|u|_p^p\geq C\int_{R_0}^{+\infty}|m(r)|^pr^{N-1}dr\geq C\int_{R_0}^{+\infty}r^{p(2-N)+N-1}dr=\infty$$
    for any $p\in (0,\frac{N}{N-2}]$, which contradicts the assumption $u\in L^p(\mathbb R^N)$ for $p\in (0,\frac{N}{N-2}]$. Thus, equation $\eqref{th1.4eq}$ has no solution. The proof is complete.
\end{proof}

\begin{proof}[The proof of Theorem \ref{th5}]
    If $b\geq S^{-2}$, then $\mathcal{P}_{0}\neq\emptyset$. Indeed, if $u\in\mathcal{P}_{0}$, we have
	$$0\leq a|\nabla u|_2^2=|u|_4^4-b|\nabla u|_2^4\leq |u|_4^4-S^{-2}|\nabla u|_2^4\leq0,$$
	which implies $u\equiv0$, contradicting to $u\in S_c$. Thus, $\mathcal{P}_{0}=\emptyset $.
	For any $u\in S_c$, one obtains$$I_0(u)=\frac{a}{2}|\nabla u|_2^2+\frac{b}{4}|\nabla u|_2^4-\frac{1}{4}|u|_4^4\geq\frac{a}{2}|\nabla u|_2^2+\frac{1}{4}(b-S^{-2})|\nabla u|_2^4>0,$$
	which implies $\inf_{u\in S_c}I_0(u)\geq0$. We also know that $s*u\in S_c$ for any $s\in \mathbb R$, then $$I_0(s*u)=\frac{ae^{2s}}{2}|\nabla u|_2^2+\frac{be^{4s}}{4}|\nabla u|_2^4-\frac{e^{4s}}{4}|u|_{4}^{4}\to 0 \ \ \rm{as} \ \ s\to-\infty.$$ Thus, $\inf_{u\in S_c}I_0(u)=0$. Using again the fact that $I_0(u)>0$ for any $u\in S_c$, we know that $\inf_{u\in S_c}I_0(u)$ can not be achieved.
	
	If $0<b<S^{-2}$,  we know that $\inf_{u\in S_c} I(u)=-\infty$ by Theorem 1.1(a) in Li et al. \cite{li2019existence}.
	In this case, setting that $S^+_c:=\{u\in S_c: |u|_4^4-b|\nabla u|_2^4>0\}$. Similarly, we can prove that $\mathcal{P}_{0}=\mathcal{P}^-_{0}$ and
	\begin{equation*}
		\inf_{u\in\mathcal{P}^-_{0}} I_0(u)=\inf_{u\in\mathcal{P}_{0}} I_0(u)=\inf_{u\in S_c^+}\max_{s\in\mathbb R} I_0(s\ast u).
	\end{equation*}
	Moreover,
    \begin{equation*}
	\begin{split}
			\inf_{u\in S_c^+}\max_{s\in\mathbb R} I_0(s\ast u)
			&=\inf_{u\in S^+_c}\frac{a^2}{4}\frac{|\nabla u|_2^4}{|u|_4^4-b|\nabla u|_2^4}\\
            &\geq\inf_{u\in H^1(\mathbb R^4)\backslash\{0\}}\frac{a^2}{4}\frac{|\nabla u|_2^4}{|u|_4^4-b|\nabla u|_2^4}=\frac{a^2}{4}\frac{S^2}{1-bS^2}=\Lambda.
		\end{split}
	\end{equation*}
	It remains to prove that
		$$\inf_{u\in S^+_c}\frac{a^2}{4}\frac{|\nabla u|_2^4}{|u|_4^4-b|\nabla u|_2^4}\leq \Lambda.$$
	We choose a radially symmetric function $\psi\in C_0^\infty(\mathbb R^4)$ with $0\leq \psi\leq 1$, $\psi=1$ for $x\in B_1(0)$ and $\psi=0$ for $x\in B_{2}^c(0)$. Moreover, we may choose $\psi$ to be non-increasing on $|x|$. Now we  define that \begin{equation*}
	   u_\varepsilon(x):=\psi(x)U_{\varepsilon,0 }(x)\ \ {\rm and} \ \ v_\varepsilon(x):=\frac{cu_\varepsilon(x)}{|u_\varepsilon|_2}\in S_c,
	\end{equation*}where $U_{\varepsilon,0 }(x)$ is given in \eqref{U}. We choose $\varepsilon=1$ and define $K_1:=|\nabla U_{1,0 }|_2^2$, $K_2:=|U_{1,0 }|_4^2$. From Appendix A in \cite{soave2020sobolev}, we have \begin{equation*}
	    \frac{K_1}{K_2}=S,\ \ |\nabla u_\varepsilon|_2^2=K_1+O(\varepsilon), \ \ |u_\varepsilon|_4^2=K_2+O(\varepsilon^2).
	\end{equation*} Then, let $\varepsilon$ be small enough, we can check that
	$$\frac{|v_\varepsilon|_4^4}{|\nabla v_\varepsilon|_2^4}=S^{-2}+O(\varepsilon^2)>b,$$
	which implies $v_\varepsilon\in S_c^+$. Thus,
	
		$$\inf_{u\in S^+_c}\frac{a^2}{4}\frac{|\nabla u|_2^4}{|u|_4^4-b|\nabla u|_2^4}\leq\frac{a^2}{4}\frac{|\nabla v_\varepsilon|_2^4}{|v_\varepsilon|_4^4-b|\nabla v_\varepsilon|_2^4}=\frac{a^2}{4}\left(\frac{S^2}{1-bS^2}+O(\varepsilon)\right)=\Lambda+O(\varepsilon).$$
The proof is complete.
\end{proof}

\section{The proof of Theorem \ref{th2}}\label{sec3}
First, we present the compactness result of the Palais-Smale sequences.
\begin{proposition}\label{pro1}
	Assume that $N\geq5$, $0<b<b_0$, $\mu>0$ and $2<q<2^*$. If $\{u_n\}\subset S_{c}$ is a Palais-Smale sequence for $I|_{S_c}$ at energy level $m\neq0$, with $m<c_{N,+}$ and
	$$ P(u_n)\rightarrow 0\ \ \ as \ \ n\rightarrow+\infty, $$
	where $c_{N,+}$ is defined by \eqref{th1eq1}. Then, up to subsequence, one of the following alternatives holds,
	\begin{itemize}
		\item [$(i)$] either $u_n\rightharpoonup u\neq0$ weakly in $H^1_{rad}(\mathbb R^N)$ but not strongly, where $u$ solves
		$$-\Big(a+bA^2\Big) \Delta u+\lambda u=\mu |u|^{q-2}u+|u|^{2^*-2}u \ \ in\ \ \mathbb{R}^N,$$
		for some $\lambda>0$ and $m-c_{N,+}\geq K(u):=\left(\frac{a}{2}+\frac{bA^2}{4 }\right)|\nabla u|_2^2-\frac{1}{2^*}|u|_{2^*}^{2^*}-\frac{\mu}{q}|u|_{q}^{q}$, where $A:=\lim_{n\rightarrow\infty}|\nabla u_n|_2>0$.
		\item [$(ii)$] or $u_n\rightarrow u\neq0$ strongly in $H^1_{rad}(\mathbb R^N)$ for some $u\in H^1_{rad}(\mathbb R^N)$. Moreover, $u\in S_c$, $I(u)=m$ and $u$ solves equation \eqref{p1.1}.
	\end{itemize}
\end{proposition}
\begin{proof}
	It follows from $I(u_n)\rightarrow m$ that
	\begin{equation*}
		\begin{split}
			\frac{a}{2}|\nabla u_n|^2_2&+\frac{b}{4}|\nabla u_n|^4_2\leq \frac{1}{2^*}|u_n|_{2^*}^{2^*}+\frac{\mu}{q}|u_n|^q_q+m+1\\
			&\leq \frac{1}{2^*}S^{-\frac{2^*}{2}}|\nabla u_n|_{2}^{2^*}+\frac{\mu}{q}C_q^q|\nabla u_n|^{q\delta_q}_2c^\frac{q(1-\delta_q)}{2}+m+1,
		\end{split}
	\end{equation*}
	combining with $q\delta_q<2^*<4$, which implies that $|\nabla u_n|_2\leq C$ for some $C$. Then $\{u_n\}$ is bounded in $H_{rad}^1(\mathbb R^N)$ and  for some $u\in H_{rad}^1(\mathbb R^N)$,
	$$u_n\rightharpoonup u\ {\rm in} \ H_{rad}^1(\mathbb R^N), \ \ u_n\rightarrow u\ {\rm in} \ L^q(\mathbb R^N), \ \ u_n(x)\rightarrow u(x) \  {\rm a.e. \ on}\ \mathbb R^N.$$
Since $\{u_n\}$ is a Palais-Smale sequence of $I|_{S_c}$, by the Lagrange multipliers rule, there exists $\lambda_n\in\mathbb R$ such that for any $\phi\in H_{rad}^1(\mathbb R^N)$ it holds
	\begin{equation}\label{pro1proof3.1}
	\begin{split}
	(a+b|\nabla u_n|^2_2)\int_{\mathbb R^N}\nabla u_n\nabla \phi dx-&\int_{\mathbb R^N}|u_n|^{2^*-2}u_n \phi dx-\mu\int_{\mathbb R^N}|u_n|^{q-2}u_n\phi dx\\
	&+\lambda_n \int_{\mathbb R^N}u_n\phi dx=o(1)||\phi||.
\end{split}
\end{equation}
In particular, taking $\phi=u_n$ in \eqref{pro1proof3.1}, then,
\begin{equation*}
-\lambda_nc=(a+b|\nabla u_n|^2_2)|\nabla u_n|_2^2-|u_n|_{2^*}^{2^*}-\mu|u_n|_q^{q}+o(1).
\end{equation*}
The boundedness of $\{u_n\}$ in $H_{rad}^1(\mathbb R^N)\cap L^q(\mathbb R^N)\cap L^{2^*}(\mathbb R^N)$ implies that $\lambda_n\rightarrow \lambda$ up to subsequence.
Combining with $ P(u_n)=o(1)$, we have
\begin{equation*}
\lambda_nc=\mu (1-\delta_q)|u_n|_q^{q}+o(1).
\end{equation*}
Let $n\rightarrow\infty$, then $\lambda c=\mu (1-\delta_q)|u|_q^{q}$. It follows from $\mu>0$ and $0<\delta_q<1$ that $\lambda\geq0$, and $\lambda=0$ if and only if $u=0$. If $\lambda=0$ and $u=0$, setting that $\lim_{n\rightarrow\infty}|\nabla u_n|_2=A\geq0$, we have
\begin{equation*}
aA^2+bA^4=\lim_{n\rightarrow\infty}(a|\nabla u_n|_2^2+b|\nabla u_n|^4_2)=\lim_{n\rightarrow\infty}|u_n|_{2^*}^{2^*}\leq S^{-\frac{2^*}{2}}\lim_{n\rightarrow\infty}|\nabla u_n|^{2^*}_2=S^{-\frac{2^*}{2}}A^{2^*},
\end{equation*}
which implies that $A=0$  or $\xi_-\leq A\leq \xi_+$, where $\xi_\pm$ are the roots of $f(t)=0$ by \eqref{ft}. If $\xi_-\leq A\leq \xi_+$, it follows that
\begin{equation*}
\begin{split}
	c_{N,+}> m=&\lim_{n\rightarrow\infty}I(u_n)=\lim_{n\rightarrow\infty}(I(u_n)-\frac{1}{2^*}P(u_n))\\
	&=\frac{1}{N}aA^2-\frac{N-4}{4N}bA^4\geq \frac{1}{N}a\xi_+^2-\frac{N-4}{4N}b\xi_+^4=c_{N,+},
\end{split}
\end{equation*}
which is impossible. The case $A=0$ is also impossible for $m\neq0$. So $\lambda>0$  and $u\neq0$ hold. Since $u_n\rightharpoonup u\neq0$ in $H^1_{rad}(\mathbb R^N)$, we set $\lim_{n\rightarrow\infty}|\nabla u_n|_2=A>0$. Then, by \eqref{pro1proof3.1}, we obtain
\begin{equation}\label{pro1proof3.6}
\begin{split}
	(a+bA^2)\int_{\mathbb R^N}\nabla u\nabla \phi dx-&\int_{\mathbb R^N}|u|^{2^*-2}u \phi dx-\mu\int_{\mathbb R^N}|u|^{q-2}u\phi dx+\lambda \int_{\mathbb R^N}u\phi dx=0,
\end{split}
\end{equation}
for any $\phi\in H^1_{rad}(\mathbb R^N)$, which means $u$ is a weak solution of $$-(a+bA^2) \Delta u+\lambda u=\mu |u|^{q-2}u+|u|^{2^*-2}u\ \ {\rm in}\ \ \mathbb{R}^N.$$ It follows from the Pohozaev identity that
\begin{equation}\label{pro1proof3.7}
Q_A(u):=(a+bA^2)|\nabla u|^2_2-\mu\delta_q|u|_q^q- |u|_{2^*}^{2^*}=0.
\end{equation}
Setting that $v_n:=u_n-u$, then $v_n\rightharpoonup0$ in $H_{rad}^1(\mathbb R^N)$ and $v_n\rightarrow0$ in $L^q(\mathbb R^N)$ for $2<q<2^*$. By the Br\'{e}zis-Lieb lemma,  we have $|\nabla u_n|_2^2=|\nabla v_n|_2^2+|\nabla u|_2^2+o(1)$ and
$$|u_n|_{q}^{q}=|v_n|_{q}^{q}+|u|_{q}^{q}+o(1)=|u|_{q}^{q}+o(1),\ \ |u_n|_{2^*}^{2^*}=|v_n|_{2^*}^{2^*}+|u|_{2^*}^{2^*}+o(1).$$
We can rewrite $P(u_n)=o(1)$ as
\begin{equation}\label{pro1proof3.8}
P(u_n)=(a+bA^2)|\nabla u_n|^2_2-\mu\delta_q|u|_q^q- |u_n|_{2^*}^{2^*}+o(1)=o(1).
\end{equation}
Setting that $\lim_{n\rightarrow\infty}|\nabla v_n|_2=l$, then it follows from \eqref{pro1proof3.7} and \eqref{pro1proof3.8} that
\begin{equation*}\label{pro1proof3.9}
\begin{split}
	al^2+bl^4&:=\lim_{n\rightarrow\infty}(a|\nabla v_n|_2^2+b|\nabla v_n|^4_2)\leq\lim_{n\rightarrow\infty}(a+bA^2)|\nabla v_n|^2_2\\
	&=\lim_{n\rightarrow\infty}|v_n|_{2^*}^{2^*}\leq S^{-\frac{2^*}{2}}\lim_{n\rightarrow\infty}|\nabla v_n|^{2^*}_2=S^{-\frac{2^*}{2}}l^{2^*},
\end{split}
\end{equation*}
which implies that $\xi_-\leq l\leq \xi_+$ or $l=0$.

$(i)$ If $\xi_-\leq l\leq \xi_+$, we define the function
\begin{equation*}\label{kt}
	k(t):=\frac{1}{4}at^2-(\frac{1}{2^*}
	-\frac{1}{4})S^{-\frac{2^*}{2}}t^{2^*}\ \  \ {\rm for}\ \ t\in(0,\infty).
\end{equation*}
With some careful analysis and some basic computations, we can obtain that
\begin{equation*}\label{c}
	\min_{t\in[\xi_-,\xi_+]}k(t)=k(\xi_+)=c_{N,+}.
\end{equation*}
Using the above analysis, we have
\begin{equation*}\label{pro1proof2.10}
\begin{split}
	m=\lim_{n\rightarrow\infty}I(u_n)&=
	K(u)+\lim_{n\rightarrow\infty}\left\{(\frac{a}{2}+\frac{bA^2}{4 })|\nabla v_n|_2^2-\frac{1}{2^*}|v_n|_{2^*}^{2^*}\right\}\\
	&=K(u)+\lim_{n\rightarrow\infty}\left\{\frac{1}{4}a|\nabla v_n|_2^2-(\frac{1}{2^*}-\frac{1}{4})|v_n|_{2^*}^{2^*}\right\}\\
	&\geq K(u)+\frac{1}{4}al^2-(\frac{1}{2^*}-\frac{1}{4})S^{-\frac{2^*}{2}}l^{2^*}\\
	&\geq K(u)+\min_{t\in[\xi_-,\xi_+]}k(t)\\
	&= K(u)+c_{N,+},
\end{split}
\end{equation*}
where $ K(u):=\left(\frac{a}{2}+\frac{bA^2}{4 }\right)|\nabla u|_2^2-\frac{1}{2^*}|u|_{2^*}^{2^*}-\frac{\mu}{q}|u|_{q}^{q}$. In this case, $(i)$ holds.

$(ii)$ If $l=0$, we know that $u_n\rightarrow u$ in $D^{1,2}(\mathbb R^N)$ and $L^{2^*}(\mathbb{R}^N)$. Testing \eqref{pro1proof3.1} and \eqref{pro1proof3.6} with $\phi=u_n-u$ and subtracting them, one obtains
$$(a+bA^2)|\nabla(u_n-u)|_2^2-\lambda|u_n-u|_2^2\to 0\ \ {\rm as}\ \ n\to\infty,$$ which implies that $u_n\to u$ in $L^2(\mathbb R^N)$, then $u_n\to u$ in $H_{rad}^1(\mathbb R^N)$. In this cases, $(ii)$ holds.
\end{proof}

\begin{lemma}\label{lemma3.2}
	There exists $\mu_1>0$ such that $\mu c^\frac{q(1-\delta_q)}{2}<\mu_1$. If the case (i)  in Proposition \ref{pro1} holds and $b_1<b<b_0$, then $K(u)\geq -c_{N,+}$.
\end{lemma}
\begin{proof}
        For any $t>0$, we define
	\begin{equation}\label{f1}
		f_1(t):=f(t)t^2-\mu
		c^\frac{q(1-\delta_q)}{2}C_q^q\delta_qt^{q\delta_q}=bt^4-S^{-\frac{2^*}{2}}t^{2^*}+at^2-\mu
		c^\frac{q(1-\delta_q)}{2}C_q^q\delta_qt^{q\delta_q}.
	\end{equation}
	Since $N\geq5$ and $2<q<2+\frac{4}{N}$, we have $0<q\delta_q<2$ and $2^*<4$, let $\xi_0^\mu$, $\xi_\pm^\mu$ be the positive roots of $f_1(t)=0$ and satisfy that
	\begin{equation*}
		0<\xi_0^\mu<\xi_-^\mu<\xi_+^\mu.
	\end{equation*}
	Recalling that $\xi_\pm$ are the roots of $f(t)=0$ in \eqref{ft}. It is evident to check that
	$$\xi_0^\mu<\xi_-^\mu<\xi_-<\xi_+<\xi_+^\mu.$$
	Moreover,
	\begin{equation}\label{pro2proof3.12}
		\xi_0^\mu\rightarrow0^+\ \ {\rm and}\ \ \xi_\pm^\mu\rightarrow\xi_\pm \ \ {\rm as}\ \ \mu c^\frac{q(1-\delta_q)}{2}\rightarrow0^+.
	\end{equation}
It follows from \eqref{pro1proof3.7} that
	\begin{equation*}
		\begin{split}
			a|\nabla u|^2_2+b|\nabla u|^4_2&< (a+bA^2)|\nabla u|^2_2= |u|_{2^*}^{2^*}+\mu\delta_q|u|_q^q\\
			&\leq S^{-\frac{2^*}{2}}|\nabla u|_{2}^{2^*}+\mu c^\frac{q(1-\delta_q)}{2}
			C_q^q\delta_q|\nabla u|^{q\delta_q}_2,
		\end{split}
	\end{equation*}
	which implies $f_1(|\nabla u|_2)<0$.
	By the above setting, we know that
	\begin{equation}\label{pro2proof3.13}
		0<|\nabla u|_2<\xi_0^\mu\ {\rm or }\  \xi_-^\mu<|\nabla u|_2<\xi_+^\mu.
	\end{equation}
	Then $b_1<b<b_0$ implies $c_{N,+}>0$. By \eqref{pro2proof3.12} and \eqref{pro2proof3.13}, one obtains
	\begin{equation}\label{eqlem3.2}
		\begin{split}
			K(u)&=\left(\frac{a}{2}+\frac{bA^2}{4 }\right)|\nabla u|_2^2-\frac{1}{2^*}|u|_{2^*}^{2^*}-\frac{\mu}{q}|u|_{q}^{q}-\frac{1}{4}Q_A(u)\\
			&=\frac{a}{4}|\nabla u|_2^2-(\frac{1}{2^*}-\frac{1}{4})|u|_{2^*}^{2^*}
			-(\frac{1}{q\delta_q}-\frac{1}{4})\delta_q\mu|u|_q^q\\
			&\geq\frac{a}{4}|\nabla u|_2^2-(\frac{1}{2^*}-\frac{1}{4})S^{-\frac{2^*}{2}}|\nabla u|_{2}^{2^*}
			-(\frac{1}{q\delta_q}-\frac{1}{4})\mu\delta_qC_q^qc^\frac{q(1-\delta_q)}{2}|\nabla u|_2^{q\delta_q}\\
			&\geq \min_{0<t<\xi_+^\mu}\left\{\frac{a}{4}t^2-(\frac{1}{2^*}
			-\frac{1}{4})S^{-\frac{2^*}{2}}t^{2^*}\right\}-(\frac{1}{q\delta_q}-\frac{1}{4})\mu\delta_qC_q^qc^\frac{q(1-\delta_q)}{2}
			(\xi_+^\mu)^{q\delta_q}\\
			&\geq0 -c_{N,+}=-c_{N,+}
		\end{split}
	\end{equation}
	for some $\mu c^\frac{q(1-\delta_q)}{2}<\mu_1$ small enough.
\end{proof}

Consider the constrained functional $I|_{S_c}$, by \eqref{S} and \eqref{G-N}, we have
\begin{equation*}
	\begin{aligned}
		I(u)&=\frac{a}{2}|\nabla u|_2^2+\frac{b}{4}|\nabla u|_2^4-\frac{\mu}{q}|u|_q^q-\frac{1}{2^*}|u|_{2^*}^{2^*}\\&\ge \frac{a}{2}|\nabla u|_2^2+\frac{b}{4}|\nabla u|_2^4-\frac{\mu}{q}C_q^qc^\frac{q(1-\delta_q)}{2}|\nabla u|_2^{q\delta_q}-\frac{S^{-\frac{2^*}{2}}}{2^*}|\nabla u|_2^{2^*}:=h_1(|\nabla u|_2).
	\end{aligned}
\end{equation*}
Hence $f_1(t)=th_1'(t)$, where $f_1(t)$ is defined by \eqref{f1}. Obviously, $\xi_-^\mu$ is the maximum point of $h_1(t)$, and $\xi_0^\mu$, $\xi_+^\mu$ are the minimum points of $h_1(t)$. Set that
\begin{equation*}
	m(b,\mu)=m(c,b,\mu):=\inf_{u\in S_c}I(u).
\end{equation*}
\begin{lemma}\label{lemma3.3}
	There exists $\mu_2>0$ such that  $\mu c^\frac{q(1-\delta_q)}{2}<\mu_2$. If $b_1<b<b_0
    $, there holds
    \begin{itemize}
        \item [$(i)$] $h_1(\xi_+^\mu)>0$,
        \item[$(ii)$] $h_1(t)$ has a unique zero point $\xi_0^{\mu,1}$ and $h_1(t)\geq0$ for $t\geq\xi_0^{\mu,1}$.
    \end{itemize}
\end{lemma}
\begin{proof}
    $(i)$ Since $h'(t)=tf(t)$, we know that $\xi_+$ is the minimum point of $h(t)$ and $h(t)>h(\xi_+)$ for $t>\xi_+$. Then, combining with $c_{N,+}>0$ for $b_1<b<b_0$, $\xi_+^\mu>\xi_+$ and \eqref{h_xi_+}, one obtains
    \begin{equation}\label{eqlem3.3}
    \begin{split}
         h_1(\xi_+^\mu)&=\frac{a}{2}(\xi_+^\mu)^2+\frac{b}{4}(\xi_+^\mu)^4-\frac{\mu}{q}C_q^qc^\frac{q(1-\delta_q)}{2}(\xi_+^\mu)^{q\delta_q}-\frac{S^{-\frac{2^*}{2}}}{2^*}(\xi_+^\mu)^{2^*}\\
         &=h(\xi_+^\mu)-\frac{\mu}{q}C_q^qc^\frac{q(1-\delta_q)}{2}(\xi_+^\mu)^{q\delta_q}>h(\xi_+)-\frac{\mu}{q}C_q^qc^\frac{q(1-\delta_q)}{2}(\xi_+^\mu)^{q\delta_q}\\
         &=c_{N,+}-\frac{\mu}{q}C_q^qc^\frac{q(1-\delta_q)}{2}(\xi_+^\mu)^{q\delta_q}>0
    \end{split}
    \end{equation}
    for some $\mu c^\frac{q(1-\delta_q)}{2}<\mu_2$ small enough. By \eqref{eqlem3.2}, we also know that $\mu_2<\mu_1$.

    $(ii)$ By the $h_1(0^+)=0^-$, $h_1(+\infty)=+\infty$ and \eqref{eqlem3.3}, it is easy to check that $h_1(t)$ has a unique zero point $\xi_0^{\mu,1}$ and $h_1(t)\geq0$ for $t\geq\xi_0^{\mu,1}$.

\end{proof}
\begin{lemma}\label{lemma3.4}For $\mu c^\frac{q(1-\delta_q)}{2}<\mu_2$ and $b_1<b<b_0$, there holds $$\inf_{u\in A_{\xi_0^{\mu,1}}}I(u)=m(b,\mu)<0,$$ where $A_{\xi_0^{\mu,1}}:=\{u\in S_c: |\nabla u|_2<\xi_0^{\mu,1}\}$.
\end{lemma}
\begin{proof}
	For any $u\in S_c$, since $\lim\limits_{s\to -\infty}I(s\ast u)=0^-$ and $\lim\limits_{s\to +\infty}I(s\ast u)=+\infty$, we can choose a small $s_0\in\mathbb R$ such that $I(s_0*u)<0$. Therefore, $m(b,\mu)<0$. By Lemma \ref{lemma3.3} $(ii)$, for any $u\in A_{\xi_0^{\mu,1}}^c$, we have $I(u)\geq h_1(|\nabla u|_2)\geq0$, which implies $\inf_{u\in A_{\xi_0^{\mu,1}}}I(u)=m(b,\mu)<0$.

\end{proof}

\begin{proof}[The proof of Theorem \ref{th2}]
$(i)$ Let $\{v_n\}$ be a minimizing sequence for $ m(b,\mu)$. From Section 3.3 and Lemma 7.17 in \cite{lieb2001analysis}, we have $I(|v_n|^*)\leq I(|v_n|)$. So we can assume that $v_n\in S_c$ is nonnegative and radially decreasing for every $n$. Then we take $s_n\in \mathbb R$ such that $s_n*v_n\in \mathcal{P}$, and
\begin{equation*}
	I(s_n*v_n)=\min\{I(s*v_n):s\in \mathbb{R}\}\leq I(v_n).
\end{equation*}
Consequently, we obtain a new minimizing sequence $\{w_n=s_n*v_n\}$ for $m(b,\mu)$, with
\begin{equation*}
	w_n\in S_c\cap \mathcal{P} \ \ {\rm and} \ P(w_n)=0
\end{equation*}
for every $n$. Hence, the Ekeland's variational principle guarantees the existence of a new minimizing sequence $\{u_n\}\subset S_c$ for $m(b,\mu)$,  with the property that $|u_n-w_n|_{H^1}\to 0$ as $n\to +\infty$, which is also a Palais-Smale sequence for $I$ on $S_c$ with $P(u_n)\to 0$ as $n\to\infty$.
Thus, $\{u_n\}$ satisfies all the assumptions of Proposition \ref{pro1}. We shall show that the alternative $(ii)$ in Proposition \ref{pro1} occurs. Otherwise, by Lemma \ref{lemma3.2} and Lemma \ref{lemma3.4}, we have
\begin{equation*}
	0>m(b,\mu)\ge K(u)+c_{N,+}\ge 0.
\end{equation*}
Consequently, up to a subsequence $u_n\to u$ in $H^1_{rad}(\mathbb{R}^N)$ and $u$ solves equation \eqref{p1.1} for some $\lambda>0$. Moreover, $u$ is nonnegative and radially decreasing and the strong maximum principle implies $u>0$.

$(ii)$ Let $u_{b,\mu}\in S_c$ be a global minimizer of $I(u)$ given by $(i)$. By Lemma \ref{lemma3.4}, we also have $u_{b,\mu}\in A_{\xi_0^{\mu,1}}$. It is easy to check that $\xi_0^{\mu,1}\to0$ as $\mu\to0^+$, and then $|\nabla u_{b,\mu}|_2\to0$ as well. From \eqref{S} and \eqref{G-N}, we have
    \begin{align*}
        0&>m(b,{\mu})=I(u_{b,\mu})=\frac{a}{2}|\nabla u_{b,\mu}|_2^2+\frac{b}{4}|\nabla u_{b,\mu}|_2^4-\frac{\mu}{q}|u_{b,\mu}|_q^q-\frac{1}{2^*}|u_{b,\mu}|_{2^*}^{2^*}\\
    &\geq\frac{a}{2}|\nabla u_{b,\mu}|_2^2+\frac{b}{4}|\nabla u_{b,\mu}|_2^4-\frac{\mu}{q}C_q^qc^\frac{q(1-\delta_q)}{2}|\nabla u_{b,\mu}|_2^{q\delta_q}-\frac{S^{-\frac{2^*}{2}}}{2^*}|\nabla u_{b,\mu}|_2^{2^*}\to0\ \ {\rm as} \ \ \mu\to0^+,
    \end{align*}
    which implies that $m(b,\mu)\to0$ as $\mu\to 0^+$. The proof is complete.
\end{proof}

\section{The proof of Theorem \ref{th6} and Theorem \ref{th7}}\label{sec6}
\subsection{The proof of Theorem \ref{th6}}

In the subsection, we study the existence of the local minimizer for equation \eqref{p1.1} with $N=4$ and hence complete the proof of Theorem \ref{th6}. For any $u\in \bar{S}_c$, by $0<b<S^{-2}$, \eqref{S} and \eqref{G-N}, we have
\begin{equation}\label{J_I_inequality}
    \begin{aligned}
		J(u)&\geq I(u)=\frac{a}{2}|\nabla u|_2^2+\frac{b}{4}|\nabla u|_2^4-\frac{\mu}{q}  |u|_q^q-\frac{1}{4}|u|_{4}^{4}\\&\ge \frac{a}{2}|\nabla u|_2^2-\frac{\mu C_q^q}{q}c^{\frac{4-q}{2}}|\nabla u|_2^{2q-4}-\frac{1-bS^2}{4S^2}|\nabla u|_2^{4}=|\nabla u|_2^2f_c(|\nabla u|_2^2),
	\end{aligned}
\end{equation} where the function $f_c(k)$ is defined by
\begin{equation}\label{f}
    f_c(k):=\frac{a}{2}-\frac{\mu C_q^q}{q}c^{\frac{4-q}{2}}k^{q-3}-\frac{1-bS^2}{4S^2}k,\ \ \forall \ c,k\in (0,+\infty).
\end{equation}In the following, similar to Lemma 2.1 in \cite{jeanjean2022orbital}, we give some results of $f_c(k)$.
\begin{lemma}\label{lem3.1}
    Assume that $0<b<S^{-2}$, $2<q<3$ and $\mu>0$. For any $c>0$ fixed, then the function $f_c(k)$ has a unique global maximum and the maximum value satisfies
    \begin{equation}\label{maxf}
        \max_{k\in(0,+\infty)}f_c(k)=f_c(k_0)=\begin{cases}
		>0, &{\rm if} \ \ c<c_0;\\
		=0, &{\rm if} \ \ c=c_0;\\
        <0, &{\rm if} \ \ c>c_0,
	\end{cases}
    \end{equation}
    where $c_0$ is defined in \eqref{c0} and \begin{equation}
        k_c:=\left[\frac{4\mu C_q^qS^2(3-q)}{q(1-bS^2)}\right]^{\frac{1}{4-q}}c^{\frac{1}{2}}.
    \end{equation} In particular, we have $k_{c_0}=k_0$.
\end{lemma}

\begin{lemma}\label{lem3.2}
    Assume that $0<b<S^{-2}$, $2<q<3$ and $\mu>0$. Let $c_1>0$ and $k_1>0$ be such that $f_{c_1}(k_1)\geq 0$. Then for any $c_2\in(0,c_1]$, it holds that
    \begin{equation}
        f_{c_2}(k_2)\geq 0, \ \ \forall \ k_2\in [\frac{c_2}{c_1}k_1,k_1].
    \end{equation}
\end{lemma}
\begin{proof}
    It is easy to see that $f_c(k)$ is non-increasing with respect to $c$. Then, for any $c_2\in(0,c_1]$, we have \begin{equation}
        f_{c_2}(k_1)\geq f_{c_1}(k_1)\geq0.
    \end{equation}
    By direct calculations, one obtains that \begin{equation}
        f_{c_2}(\frac{c_2}{c_1}k_1)\geq f_{c_1}(k_1)\geq0.
    \end{equation} If $f_{c_2}(k)<0$ for some $k\in (\frac{c_2}{c_1}k_1,k_1)$, then $f_{c_2}(k)$ has a local minimum point in $(\frac{c_2}{c_1}k_1,k_1)$, which contradicts the fact $f_{c_2}(k)$ has a unique global maximum by Lemma \ref{lem3.1}. The proof is complete.
\end{proof}

\begin{lemma}\label{lem3.3}
    Assume that $0<b<S^{-2}$, $2<q<3$ and $\mu>0$. For any $c\in (0,c_0)$, it holds that
\begin{equation}\label{m}
    m(c):=m(c,b,\mu)=\inf_{u\in A_{k_0}(c)}I(u)<0<\inf_{u\in \partial{(A_{k_0}(c)})}I(u)
\end{equation}and
\begin{equation}\label{bm}
    \bar{m}(c):=\bar{m}(c,b,\mu)=\inf_{u\in A_{k_0}(c)}J(u)<0<\inf_{u\in \partial{(A_{k_0}(c)})}J(u).
\end{equation}
\end{lemma}
\begin{proof}
    For any $u\in \bar{S}_c$, set that $u_t(x):=t^2u(tx)$. Then we can check that $u_t\in \bar{S}_c$ and $|\nabla u_t|_2^2=t^2|\nabla u|_2^2$. Moreover, we have
    \begin{equation*}
        I(u_t)\leq J(u_t)=\frac{at^2}{2(1-bS^2)}|\nabla u|^2_2+\frac{bt^4}{4(1-bS^2)}|\nabla u|^4_2-\frac{\mu t^{2q-4}}{q}|u|^q_q-\frac{t^4}{4}|u|^{4}_{4}.
    \end{equation*} Since $2<q<3$, there exists $t_0>0$ small enough such that $|\nabla u_{t_0}|_2^2<k_0$ and $I(u_{t_0})\leq J(u_{t_0})<0$, which implies that $u_{t_0}\in A_{k_0}(c)$ and $m(c)\leq \bar{m}(c)<0$. Moreover, by \eqref{J_I_inequality} and the fact that $f_c(k_0)>f_{c_0}(k_0)=0$ for any $c\in(0,c_0)$, we have $J(u)\geq I(u)\geq |\nabla u|_2^2f_c(|\nabla u|_2^2)=k_0f_c(k_0)>0$ for any $u\in \partial{(A_{k_0}(c))}$. Thus, \eqref{m} and \eqref{bm} hold. The proof is complete.
\end{proof}

\begin{lemma} \label{lem3.4}
Assume that $0<b<S^{-2}$, $2<q<3$ and $\mu>0$. It holds that
\begin{itemize}
    \item [$(i)$] Let $c\in(0,c_0)$. For any $\alpha\in(0,c)$, we have $m(c)\leq m(\alpha)+m(c-\alpha)$. Moreover, if $m(\alpha)$ or $m(c-\alpha)$ is achieved, then the inequality is strict.
    \item [$(ii)$] The function $c\mapsto m(c) $ is continuous on $(0,c_0)$.
\end{itemize}
\end{lemma}

The proof of Lemma \ref{lem3.4} is similar to the following lemma, so we omit it.

\begin{lemma} \label{lem3.5}
Assume that $0<b<S^{-2}$, $2<q<3$ and $\mu>0$. It holds that
\begin{itemize}
    \item [$(i)$] Let $c\in(0,c_0)$. For any $\alpha\in(0,c)$, we have $\bar{m}(c)\leq \bar{m}(\alpha)+\bar{m}(c-\alpha)$. Moreover, if $\bar{m}(\alpha)$ or $\bar{m}(c-\alpha)$ is achieved,  then the inequality is strict.
    \item [$(ii)$] The function $c\mapsto \bar{m}(c) $ is continuous on $(0,c_0)$.
\end{itemize}
\end{lemma}
\begin{proof}
    $(i)$ Let $\{u_n\}\in A_{k_0}(\alpha)$ be a minimizing sequence with respect to $\bar{m}(\alpha)$, i.e., $\lim_{n\to \infty}J(u_n)=\bar{m}(\alpha)$. By Lemma \ref{lem3.3}, we have $\bar{m}(\alpha)<0$ and $0>J(u_n)\geq|\nabla u_n|_2^2f_\alpha(|\nabla u_n|_2^2)$ for $n\in \mathbb N$ large enough. Then, in view of Lemma \ref{lem3.2} and the fact that $f_c(k_0)>f_{c_0}(k_0)=0$ for any $c\in(0,c_0)$, we have $f_{\alpha}(k)\geq 0$ for any $k\in [\frac{\alpha}{c}k_0,k_0]$, which implies that \begin{equation*}
        |\nabla u_n|_2^2<\frac{\alpha}{c}k_0
    \end{equation*} for $n\in \mathbb N$ large enough. For any $\theta \in (1,\frac{c}{\alpha}]$, set that $v_n(x):=u_n(\theta ^{-\frac{1}{4}}x)$. Then we have $|v_n|_2^2=\theta|u_n|_2^2=\theta\alpha$ and \begin{equation*}
        |\nabla v_n|_2^2=\theta^{\frac{1}{2}}|\nabla u_n|_2^2\leq (\frac{c}{\alpha})^{\frac{1}{2}}\frac{\alpha}{c}k_0<k_0,
    \end{equation*} which means that $v_n\in A_{k_0}(\theta\alpha)$. Thus, it follows from \eqref{bm} that \begin{equation}\label{m(alpha)}
    \begin{aligned}
          \bar{m}(\theta\alpha)&\leq J(v_n) \\ & =\frac{a\theta^{\frac{1}{2}}}{2(1-bS^2)}|\nabla u_n|^2_2+\frac{b\theta}{4(1-bS^2)}|\nabla u_n|^4_2-\frac{\mu\theta}{q}|u_n|^q_q-\frac{\theta}{4}|u_n|^{4}_{4}\\ &\leq\frac{a\theta}{2(1-bS^2)}|\nabla u_n|^2_2+\frac{b\theta}{4(1-bS^2)}|\nabla u_n|^4_2-\frac{\mu\theta}{q}|u_n|^q_q-\frac{\theta}{4}|u_n|^{4}_{4}\\&=\theta J(u_n)=\theta \bar{m}(\alpha)+o(1),
    \end{aligned}
    \end{equation} which implies that \begin{equation*}
        \bar m(\theta\alpha)\leq \theta\bar{m}(\alpha), \ \ \forall \ \theta\in(1,\frac{c}{\alpha}].
    \end{equation*} If $\bar{m}(\alpha)$ is achieved by $u\in A_{k_0}(\alpha)$, then we can choose $u_n \equiv u$ in \eqref{m(alpha)} and hence the strict inequality holds. Moreover, it follows that
    $$\bar{m}(c)=\frac{c-\alpha}{c}\bar{m}(c)+\frac{\alpha}{c}\bar{m}(c)\leq \bar{m}(c-\alpha)+\bar{m}(\alpha)$$ with a strict inequality if $\bar{m}(c-\alpha)$ and $\bar{m}(\alpha)$ is achieved.

    $(ii)$ Let $\{\bar{c}_n\}\subset(0,c_0)$ be such that $\lim_{n\to\infty} \bar{c}_n=c\subset(0,c_0)$. For any $\alpha\in (0,c_0)$, Lemma \ref{lem3.3} implies that $\bar{m}(\alpha)<0$. If $\bar{c}_n<c$, then it follows from $(i)$ that \begin{equation}\label{m_cn_inequality}
        \bar{m}(c)\leq \bar{m}(\bar{c}_n)+\bar{m}(c-\bar{c}_n)<\bar{m}(\bar{c}_n).
    \end{equation} If $\bar{c}_n\geq c$, we let $u_n\in A_{k_0}(\bar{c}_n)$ be such that $J(u_n)\leq \bar{m}(\bar{c}_n)+\frac{1}{n}$. Set that $v_n=\sqrt{\frac{c}{\bar{c}_n}}u_n$. Then $v_n\in A_{k_0}(c)$. Moreover, we have \begin{equation*}
    \begin{aligned}
        \bar{m}(c)&\leq J(v_n)=J(u_n)+[J(v_n)-J(u_n)]\\&=J(u_n)+\frac{a}{2(1-bS^2)}\frac{c-\bar{c}_n}{\bar{c}_n}|\nabla u|^2_2+\frac{b}{4(1-bS^2)}\frac{c^2-\bar{c}_n^2}{\bar{c}_n^2}|\nabla u|^4_2\\&-\frac{\mu(c^{\frac{q}{2}}-\bar{c}_n^{\frac{q}{2}})}{q\bar{c}_n^{\frac{q}{2}}}|u|^q_q-\frac{c^2-\bar{c}_n^2}{4\bar{c}_n^2}|u|^{4}_{4}\\&=J(u_n)+o(1)\\&\leq \bar{m}(\bar{c}_n)+o(1).
    \end{aligned}
    \end{equation*} Combining this with \eqref{m_cn_inequality}, we have \begin{equation}\label{mm1}
        \bar{m}(c)\leq \bar{m}(\bar{c}_n)+o(1).
    \end{equation} For $\varepsilon>0$ small, there exists $u\in A_{k_0}(c)$ such that $J(u)<\bar{m}(c)+\varepsilon$.
    Set that $w_n=\sqrt{\frac{\bar{c}_n}{c}}u$. Then we have $w_n\in A_{k_0}(\bar{c}_n)$ for $n\in \mathbb N$ large enough. Since $\lim_{n\to \infty}J(w_n)=J(u)$, one obtains that
    $\bar{m}(\bar{c}_n)\leq J(w_n)=J(u)+[J(w_n)-J(u)]=J(u)+o(1)<\bar{m}(c)+\varepsilon+o(1)$. Thus, since $\varepsilon>0$ is arbitrary, we have \begin{equation}\label{mm2}
        \bar{m}(\bar{c}_n)<\bar{m}(c)+o(1).
    \end{equation} Both \eqref{mm1} and \eqref{mm2} conclude that $\lim_{n\to \infty}\bar{m}(\bar{c}_n)=\bar{m}(c)$. The proof is complete.
\end{proof}
\begin{proof}[The proof of Theorem \ref{th6}]
  $(i)$ Let $\{u_n\}\subset A_{k_0}(c)$ be a minimizing sequence with respect to $m(c)$. Since  $\{|u_n|\}\subset A_{k_0}(c)$ is also a minimizing sequence with respect to $m(c)$, so we can assume that $u_n\geq0$. It follows from Lemma \ref{lem3.3} that \begin{equation}\label{I(un)}
        |u_n|_2^2=c,\ \ |\nabla u_n|_2^2<k_0<+\infty,\ \ I(u_n)=m(c)+o(1)<0.
    \end{equation}
Let $$\delta:=\lim\sup_{n\to\infty}\sup_{y\in\mathbb R^4}\int_{B_1(y)}|u_n|^2dx\geq0.$$ If $\delta=0$, then it follows from Lions' concentration compactness principle in \cite{minimax} that $u_n\to 0$ in $L^s(\mathbb R^4)$ for $s\in(2,4)$. Then, together with \eqref{S}, \eqref{f}, \eqref{maxf} and \eqref{I(un)}, we have \begin{equation*}
        \begin{aligned}
            m(c)+o(1)=I(u_n)&=\frac{a}{2}|\nabla u_n|_2^2+\frac{b}{4}|\nabla u_n|_2^4-\frac{1}{4}|u_n|_4^4-\frac{\mu}{q}|u_n|_q^q\\&=\frac{a}{2}|\nabla u_n|_2^2+\frac{b}{4}|\nabla u_n|_2^4-\frac{1}{4}|u_n|_4^4+o(1)\\&\geq\frac{a}{2}|\nabla u_n|_2^2-\frac{1-bS^2}{4S^2}|\nabla u_n|_2^4+o(1)\\&\geq|\nabla u_n|_2^2\left(\frac{a}{2}-\frac{1-bS^2}{4S^2}k_0\right)+o(1)\\&=|\nabla u_n|_2^2\left[f_c(k_0)+\frac{\mu C_q^q}{q}c^{\frac{4-q}{2}}k_0^{q-3}\right]+o(1)\\&\geq o(1),
        \end{aligned}
    \end{equation*}which contradicts the fact that $m(c)<0$. Thus, $\delta>0$ and there exists a sequence $\{y_n\}\in\mathbb R^4$ such that, up to a subsequence, $\int_{B_1(y_n)}|u_n|^2dx>\frac{\delta}{2}.$ Set that $\tilde{u}_n:=u_n(x+y_n)$. Then $\int_{B_1(0)}|\tilde{u} _n|^2dx>\frac{\delta}{2}.$ Hence, there exists a $\tilde{u}\in H^1(\mathbb R^4)\setminus\{0\}$ with $\tilde{u}\geq 0$ such that, up to a subsequence, \begin{equation}\label{tu_n}
        \begin{cases}
		\tilde{u}_n\rightharpoonup \tilde{u} &{\rm in} \ \ H^1(\mathbb R^4);\\
		\tilde{u}_n\to\tilde{u} &{\rm in} \ \ L^s_{loc}(\mathbb R^4), \ \ \forall \ s\in(2,4);\\
        \tilde{u}_n(x)\to\tilde{u}(x) &{\rm a.e. \ on} \ \ \mathbb R^4.
	\end{cases}
    \end{equation} Furthermore, by the variance of translations and \eqref{I(un)}, we have \begin{equation}\label{I(tun)}
        0<|\tilde{u}|_2^2\leq|\tilde{u}_n|_2^2=c,\ \ |\nabla \tilde{u}_n|_2^2<k_0<+\infty,\ \ I(\tilde{u}_n)=m(c)+o(1).
    \end{equation}
Set that $v_n:=\tilde{u}_n-\tilde{u}$. By \eqref{tu_n} and the Brezis-Lieb Lemma, one obtains that \begin{equation}\label{ntu}
        |\nabla\tilde{u}_n|_2^2=|\nabla\tilde{u}|_2^2+|\nabla v_n|_2^2+o(1), \ \ |\nabla\tilde{u}_n|_2^4=|\nabla\tilde{u}|_2^4+|\nabla v_n|_2^4+2|\nabla\tilde{u}|_2^2|\nabla v_n|_2^2+o(1)
    \end{equation}and \begin{equation}\label{bl}
        c=|\tilde{u}_n|_2^2=|\tilde{u}|_2^2+|v_n|_2^2+o(1),\ \ |\tilde{u}_n|_4^4=|\tilde{u}|_4^4+|v_n|_4^4+o(1).
    \end{equation} It follows that \begin{equation}\label{IuIv}
        m(c)=I(\tilde{u}_n)+o(1)=I(\tilde{u})+I(v_n)+\frac{b}{2}|\nabla\tilde{u}|_2^2|\nabla v_n|_2^2+o(1).
    \end{equation}
    Now, we prove that $v_n\to0$ in $H^1(\mathbb R^4)$ and henceforth $\tilde{u}_n\to\tilde{u}$ in $H^1(\mathbb R^4)$. This proof is divided into two steps:

    Firstly, we claim that $|v_n|_2^2\to 0$ as $n\to\infty$, i.e., $|\tilde{u}|_2^2=c$. Suppose by contradiction that $|\tilde{u}|_2^2=:\tilde{c}<c$. From \eqref{ntu} and \eqref{bl}, for $n\in\mathbb N$ large enough, we have \begin{equation*}
        \alpha_n:=|v_n|_2^2\leq c, \ \ |\nabla v_n|_2^2\leq |\nabla\tilde{u}_n|_2^2<k_0,
    \end{equation*}which implies that $v_n\in A_{k_0}(\alpha_n)$ and $I(v_n)\geq m(\alpha_n)$. Then, together with \eqref{IuIv}, one obtains that \begin{equation}
         m(c)=I(\tilde{u})+I(v_n)+\frac{b}{2}|\nabla\tilde{u}|_2^2|\nabla v_n|_2^2+o(1)\geq I(\tilde{u})+m(\alpha_n)+o(1).
    \end{equation} By Lemma \ref{lem3.5} $(ii)$ and \eqref{bl}, we have \begin{equation}\label{mIm}
        m(c)\geq I(\tilde{u})+m(c-\tilde{c}).
    \end{equation} We also have that $\tilde{u}\in\overline{A_{k_0}(\tilde{c})}$ by the weak limit, which implies that $I(\tilde{u})\geq m(\tilde{c})$. If $I(\tilde{u})>m(\tilde{c})$, then it follows from \eqref{mIm} and Lemma \ref{lem3.5} $(i)$ that $$m(c)> m(\tilde{c})+m(c-\tilde{c})\geq m(c),$$which is a contradiction. If $I(\tilde{u})=m(\tilde{c})$, then the strict inequality in Lemma \ref{lem3.5} $(i)$ holds and $$m(c)\geq m(\tilde{c})+m(c-\tilde{c})>m(c),$$which is also a contradiction. Thus, $|\tilde{u}|_2^2=c$ and $\tilde{u}\in\overline{A_{k_0}(c)}$ by the weak limit.

    Secondly, we claim that $|\nabla v_n|_2^2\to 0$ as $n\to\infty$. It follows from \eqref{f}, \eqref{maxf}, \eqref{IuIv} and $I(\tilde{u})\geq m(c)$ that \begin{equation*}
    \begin{aligned}
         o(1)&\geq\frac{a}{2}|\nabla v_n|_2^2+\frac{b}{4}|\nabla v_n|_2^4-\frac{1}{4}|v_n|_4^4+\frac{b}{2}|\nabla\tilde{u}|_2^2|v_n|_2^2+o(1)\geq\frac{a}{2}|\nabla v_n|_2^2-\frac{1-bS^2}{4S^2}|\nabla v_n|_2^4+o(1)\\
         &\geq|\nabla v_n|_2^2\left(\frac{a}{2}-\frac{1-bS^2}{4S^2}k_0\right)+o(1)=|\nabla v_n|_2^2\left[f_c(k_0)+\frac{\mu C_q^q}{q}c^{\frac{4-q}{2}}k_0^{q-3}\right]+o(1),
    \end{aligned}
    \end{equation*} which implies that $|\nabla v_n|_2^2\to 0$ as $n\to \infty$. Together with the fact that $|v_n|_2^2\to0$ as $n\to\infty$, we have $v_n\to0$ in $H^1(\mathbb R^4)$, i.e., $\tilde{u}_n\to\tilde{u}$ in $H^1(\mathbb R^4)$. Thus, \begin{equation*}
       |\tilde{u}|_2^2=c,\ \ |\nabla \tilde{u}|_2^2\leq k_0, \ \ I(\tilde{u})=m(c).
    \end{equation*} Combining this with Lemma \ref{lem3.3}, we know that $|\nabla\tilde{u}|_2^2<k_0$ and henceforth $\tilde{u}\in A_{k_0}(c)$. By the Corollary 2.4 in \cite{chen2024normalized}, one obtains that $I|'_{S_c}(\tilde{u})=0$ and henceforth there exists a lagrange multiplier $\tilde{\lambda}_c\in \mathbb R$ such that $$-(a+b|\nabla\tilde{u}|_2^2)\Delta\tilde{u}+\tilde{\lambda}_c\tilde{u}=\tilde{u}^3+\mu|\tilde{u}|^{q-2}\tilde{u}.$$ It is easy to check that $\tilde{\lambda}_c>0$. Since $\tilde{u}\geq 0$ and $\tilde{u}\neq0$, the strong maximum principle implies that $\tilde{u}>0$.

     $(ii)$ Similar to $(i)$, by Lemmas \ref{lem3.3} and \ref{lem3.5}, we can prove that a positive  $\bar{u}\in A_{k_0}(c)$ satisfies that  $J(\bar{u})=\bar{m}(c)<0$ and for a $\bar{\lambda}_c>0$, \begin{equation}\label{new_equation}
        - \left [ \frac{a}{(1-bS^2)}+\frac{b}{(1-bS^2)}|\nabla \bar{u}_c|_2^2\right ]\Delta\bar{u}_c+\bar{\lambda}_c\bar{u}_c=\bar{u}_c^3+\mu|\bar{u}_c|^{q-2}\bar{u}_c.
    \end{equation}
     The proof is complete.
\end{proof}
\subsection{The proof of Theorem \ref{th7}}
In the  subsection, we establish the threshold of the mountain pass level to prove the existence of the mountain pass solution, thereby proving Theorem \ref{th7}. By a standard argument, we obtain the Pohozaev-type functional related to \eqref{new_equation} as follows.
\begin{lemma}\label{new_Pohozaev}
Assume that $0<b<S^{-2}$, $2<q<3$, $\mu>0$ and $c\in(0,c_0)$. It holds that
   \begin{equation}
    \frac{a}{(1-bS^2)}|\nabla \bar{u}_c|_2^2+\frac{b}{(1-bS^2)}|\nabla \bar{u}_c|_2^4-|\bar{u}_c|_4^4-\frac{\mu(2q-4)}{q}|\bar{u}_c|_q^q=0.
\end{equation}
\end{lemma}
\begin{lemma}\label{mountainpass}
    Assume that $0<b<S^{-2}$, $2<q<3$, $\mu>0$ and $c\in(0,c_0)$. Then there exist $d>0$ such that \begin{equation}\label{M}
        M(c):=\inf_{\gamma\in \Gamma_c}\max_{t\in[0,1]}I(\gamma(t))\geq d>\sup_{\gamma\in\Gamma_c}\max\{I(\gamma(0)),I(\gamma(1))\},
    \end{equation}where
    \begin{equation}
        \Gamma_c:=\{\gamma\in \mathcal{C}([0,1], S_c):\gamma(0)=\bar{u}_c, I(\gamma(1))<2m(c)\}.
    \end{equation}
\end{lemma}
\begin{proof}
    Set that $d:=\inf_{u\in\partial(A_{k_0}(c)) }I(u)$. By \eqref{m}, we know that $d>0$. Let $\gamma\in\Gamma_c$ be arbitrary. By Lemma \ref{lem3.3}, one obtains that $\gamma(0)=\bar{u}_c\in (A_{k_0}(c))\setminus(\partial(A_{k_0}(c)))$ and $I(\gamma(1))<2m(c)<m(c)<0$. Using again \eqref{m}, we have $\gamma(1)\notin \overline{A_{k_0}(c)}$. Since $\gamma(t)$ is continuous on $[0,1]$, there exists a $t_0\in(0,1)$ such that $\gamma(t_0)\in\partial(A_{k_0}(c))$. Then, $\max_{t\in[0,1]}I(\gamma(t))\geq d>0$. Since $I(\gamma(0))=I(\bar{u}_c)<J(\bar{u}_c)=\bar{m}(c)<0$. Thus, \eqref{M} holds. The proof is completed.
\end{proof}
Similar to Lemma 3.11 in \cite{chen2024normalized}, we obtain the following lemma.
\begin{lemma}\label{ps}
    Assume that $0<b<S^{-2}$, $2<q<3$, $\mu>0$ and $c\in(0,c_0)$. Then there exists a sequence $\{u_n\}\in S_c$ such that \begin{equation}
        I(u_n)\to M(c)>0, \ \ I|'_{S_c}(u_n)\to 0 \ \ and \ \ P(u_n)\to 0.
    \end{equation}
\end{lemma}
Now, we define the functions $U_n(x):=T_n(|x|)$, where \begin{equation}\label{T}
    T_n(r)=2\sqrt2\begin{cases}
		\sqrt{\frac{n}{1+n^2r^2}}, &0\leq r<1;\\
		\sqrt{\frac{n}{1+n^2}}(2-r),&1\leq r<2;\\
        0,&r\geq2.
	\end{cases}
\end{equation}
From \cite{chen2024another}, we have
\begin{equation}\label{U2}
    |U_n|_2^2=O\left(\frac{\log(1+n^2)}{n^2}\right), \ \ n\to \infty,
\end{equation}
\begin{equation}\label{nabla_U2}
    |\nabla U_n|_2^2=S^2+O\left(\frac{1}{n^2}\right), \ \ n\to \infty,
\end{equation}
\begin{equation}\label{U4}
    |U_n|_4^4=S^2+O\left(\frac{1}{n^4}\right), \ \ n\to \infty.
\end{equation}
In the following, we consider the superposition of $\bar{u}$ and $U_n$ to construct a test function in $S_c$. By some delicate calculations, we obtain the threshold of mountain pass level $M(c)$ as follows, which is a crucial step in proving Theorem \ref{th7}.
\begin{lemma}\label{lem3.10}
     Assume that $0<b<S^{-2}$, $2<q<3$, $\mu>0$ and $c\in(0,c_0)$. It holds that \begin{equation}\label{Mm}
         M(c)<\bar{m}(c)+\Lambda.
     \end{equation}
\end{lemma}
\begin{proof}
    Let $\bar{u}_c\in S_c$ is given by Theorem \ref{th6} $(ii)$. Then by Theorem \ref{th6} $(ii)$  and Lemma \ref{new_Pohozaev}, one obtains that \begin{equation}\label{uc}
        |\bar{u}_c|_2^2=c,\ \ J(\bar{u}_c)=\bar{m}(c), \ \ \bar{\lambda}_c|\bar{u}_c|_2^2=\frac{\mu(4-q)}{q}|\bar{u}_c|_q^q,\ \ \bar{u}_c(x)>0,\ \ \forall \ x\in \mathbb R^4
    \end{equation}and
    \begin{equation}\label{laga}
        \left [ \frac{a}{(1-bS^2)}+\frac{b}{(1-bS^2)}|\nabla \bar{u}_c|_2^2\right ]\int_{\mathbb R^4} \nabla\bar{u}_c\nabla U_n dx=\int_{\mathbb R^4}(\bar{u}_c^3+\mu\bar{u}_c^{q-1}-\bar{\lambda}_c\bar{u}_c)U_n dx.
    \end{equation} Set that $B:=\inf_{|x|\leq1}\bar{u}_c(x)$. Then $B>0$. Together with \eqref{T}, \eqref{U2} and \eqref{laga}, we can deduce that
    \begin{equation}\label{uU}
        \int_{\mathbb R^4}\bar{u}_cU_ndx=O\left(\frac{\sqrt{\log(1+n^2)}}{n}\right), \ \ n\to\infty,
    \end{equation}
    \begin{equation}\label{nunU}
        \left | \int_{\mathbb R^4}\nabla\bar{u}_c\nabla U_ndx\right |=\left | \int_{\mathbb R^4}U_n\Delta \bar{u}_cdx\right |=O\left(\frac{\sqrt{\log(1+n^2)}}{n}\right), \ \ n\to\infty,
    \end{equation}
    \begin{equation}\label{uqU}
        \int_{\mathbb R^4}\bar{u}_c^{q-1}U_ndx\leq\left [\int_{\mathbb R^4}\bar{u}_c^{2(q-1)}dx\int_{|x|\leq2}U_n^2dx\right]^{\frac{1}{2}} =O\left(\frac{\sqrt{\log(1+n^2)}}{n}\right), \ \ n\to\infty,
    \end{equation}and
    \begin{equation}\label{uU3}
    \begin{aligned}
        \int_{\mathbb R^4}\bar{u}_cU_n^3dx&\geq2\pi^2 B\int_0^1r^3|T_n(r)|^{3}dr\\&\geq32\sqrt{2}\pi^2B\int_0^1\frac{n^3r^3}{(1+n^2r^2)^3}dr\\&\geq\frac{32\sqrt{2}\pi^2B}{n}\int_0^n\frac{s^3}{(1+s^2)^3}ds:=\frac{B_0}{n}.
    \end{aligned}
    \end{equation}
    By \eqref{U2} and \eqref{uc}, we have \begin{equation}\label{uc+Un}
    \begin{aligned}
         |\bar{u}_c+tU_n|_2^2&=c+t^2|U_n|_2^2+2t\int_{\mathbb R^4}\bar{u}_cU_ndx\\&=c+2t\int_{\mathbb R^4}\bar{u}_cU_ndx+t^2\left[O\left(\frac{\log(1+n^2)}{n^2}\right)\right], \ \ n\to\infty.
    \end{aligned}
    \end{equation}Let $\tau=\tau_{n,t}:=\frac{|\bar{u}_c+tU_n|_2}{\sqrt{c}}$. Then
    \begin{equation}\label{tau}
        \tau^2=1+\frac{2t}{c}\int_{\mathbb R^4}\bar{u}_cU_ndx+t^2\left[O\left(\frac{\log(1+n^2)}{n^2}\right)\right], \ \ n\to\infty.
    \end{equation} Now, we define \begin{equation}\label{W}
        W_{n,t}:=\tau\left[\bar{u}_c(\tau x)+tU_n(\tau x)\right].
    \end{equation}Then we have \begin{equation}\label{nW}
        |\nabla W_{n,t}|_2^2=|\nabla(\bar{u}_c+tU_n)|_2^2,\ \ | W_{n,t}|_4^4=|\bar{u}_c+tU_n|_4^4
    \end{equation}and \begin{equation}\label{Wq}
        |W_{n,t}|_2^2=\tau^{-2}|\bar{u}_c+tU_n|_2^2=c,\ \ |W_{n,t}|_q^q=\tau^{q-4}|\bar{u}_c+tU_n|_q^q.
    \end{equation}Set that \begin{equation}\label{t*}
        t_*^2:=\frac{a+b|\nabla\bar{u}_c|^2_2}{1-bS^2}.
    \end{equation}Then \eqref{laga} can be rewritten as \begin{equation}
        (a+b|\nabla\bar{u}_c|^2_2+bS^2t_*^2)\int_{\mathbb R^4}\nabla\bar{u}_c\nabla U_ndx=\int_{\mathbb R^4}(\bar{u}_c^3+\mu|\bar{u}_c|^{q-2}\bar{u}_c-\bar{\lambda}_c\bar{u}_c)U_ndx.
    \end{equation} It follows from \eqref{t*} that
    \begin{equation}\label{t_*-inequality}
    \begin{aligned}
          &S^2\left[\frac{a+b|\nabla\bar{u}_c|^2_2}{2}t^2-\frac{1-bS^2}{4}t^4\right]<S^2\left[\frac{a+b|\nabla\bar{u}_c|^2_2}{2}t_*^2-\frac{1-bS^2}{4}t_*^4\right]\\
          &=S^2\left[\frac{(a+b|\nabla\bar{u}_c|^2_2)^2}{4(1-bS^2)}\right]=\frac{a^2S^2}{4(1-bS^2)}+\frac{abS^2|\nabla\bar{u}_c|^2_2}{2(1-bS^2)}+\frac{b^2S^2|\nabla\bar{u}_c|^4_2}{4(1-bS^2)}\\
          &=\Lambda+\frac{abS^2|\nabla\bar{u}_c|^2_2}{2(1-bS^2)}+\frac{b^2S^2|\nabla\bar{u}_c|^4_2}{4(1-bS^2)},\ \ \forall \ t\in(0,t_*)\cup(t_*,+\infty).
    \end{aligned}
    \end{equation}
    By direct calculations, we can check that \begin{equation}\label{eq1}
        (1+t)^p\geq 1+pt+pt^{p-1}+t^p,\ \ \forall \ p\geq3,\ t\geq0
    \end{equation}
    and \begin{equation}\label{eq2}
        (1+t)^p\geq 1+pt^{p-1}+t^p,\ \ \forall \ p\geq2,\ t\geq0.
    \end{equation}
    From \eqref{U2}-\eqref{U4}, \eqref{uc} and \eqref{uU}-\eqref{eq2}, we have \begin{equation}
    \begin{aligned}
        I(W_{n,t})&=\frac{a}{2}|\nabla W_{n,t}|_2^2+\frac{b}{4}|\nabla W_{n,t}|_2^4-\frac{1}{4}|W_{n,t}|_4^4-\frac{\mu}{q}|W_{n,t}|_q^q\\&=\frac{a}{2}|\nabla(\bar{u}_c+tU_n)|_2^2+\frac{b}{4}|\nabla(\bar{u}_c+tU_n)|_2^4-\frac{1}{4}|\bar{u}_c+tU_n|_4^4-\frac{\mu\tau^{q-4}}{q}|\bar{u}_c+tU_n|_q^q\\&\leq\frac{a}{2}|\nabla\bar{u}_c|_2^2+\frac{b}{4}|\nabla\bar{u}_c|_2^4-\frac{1}{4}|\bar{u}_c|^4_4-\frac{\mu\tau^{q-4}}{q}|\bar{u}_c|_q^q+\frac{at^2}{2}|\nabla U_n|_2^2+\frac{bt^4}{4}|\nabla U_n|_2^4\\&-\frac{t^4}{4}|U_n|_4^4-t\int_{\mathbb R^4}\bar{u}_c^3U_ndx-t^3\int_{\mathbb R^4}\bar{u}_cU_n^3dx-\mu\tau^{q-4}t\int_{\mathbb R^4}\bar{u}_c^{q-1}U_ndx+\frac{bt^2}{2}|\nabla \bar{u}_c|_2^2|\nabla U_n|_2^2\\&+\left(a+b|\nabla \bar{u}_c|_2^2+bt^2|\nabla U_n|_2^2\right)t\int_{\mathbb R^4}\nabla\bar{u}_c\nabla U_ndx+bt^2\left(\int_{\mathbb R^4}\nabla\bar{u}_c\nabla U_ndx\right)^2\\&=\frac{a}{2}|\nabla\bar{u}_c|_2^2+\frac{b}{4}|\nabla\bar{u}_c|_2^4-\frac{1}{4}|\bar{u}_c|^4_4-\frac{\mu}{q}|\bar{u}_c|_q^q+\frac{at^2}{2}|\nabla U_n|_2^2+\frac{bt^4}{4}|\nabla U_n|_2^4-\frac{t^4}{4}|U_n|_4^4\\&+\frac{\mu(1-\tau^{q-4})}{q}|\bar{u}_c|_q^q+\mu(1-\tau^{q-4})t\int_{\mathbb R^4}\bar{u}_c^{q-1}U_ndx-\bar{\lambda}_ct\int_{\mathbb R^4}\bar{u}_cU_ndx\\&+\frac{bt^2}{2}|\nabla \bar{u}_c|_2^2|\nabla U_n|_2^2+b(t^2|\nabla U_n|_2^2-t_*^2S^2)t\int_{\mathbb R^4}\nabla\bar{u}_c\nabla U_ndx\\&-t^3\int_{\mathbb R^4}\bar{u}_cU_n^3dx+t^2\left[O\left(\frac{\log(1+n^2)}{n^2}\right)\right]\\&=\frac{a}{2}|\nabla\bar{u}_c|_2^2+\frac{b}{4}|\nabla\bar{u}_c|_2^4-\frac{1}{4}|\bar{u}_c|^4_4-\frac{\mu}{q}|\bar{u}_c|_q^q+S^2\left[\frac{a+b|\nabla \bar{u}_c|_2^2}{2}t^2-\frac{1-bS^2}{4}t^4\right]\\&+\frac{\mu|\bar{u}_c|_q^q}{q} \left\{1-\left[1+\frac{2t}{c}\int_{\mathbb R^4}\bar{u}_cU_ndx+t^2\left(O\left(\frac{\log(1+n^2)}{n^2}\right)\right)\right]^{\frac{q-4}{2}}\right\}-\bar{\lambda}_ct\int_{\mathbb R^4}\bar{u}_cU_ndx\\&+\mu \left\{1-\left[1+\frac{2t}{c}\int_{\mathbb R^4}\bar{u}_cU_ndx+t^2\left(O\left(\frac{\log(1+n^2)}{n^2}\right)\right)\right]^{\frac{q-4}{2}}\right\}t\int_{\mathbb R^4}\bar{u}_c^{q-1}U_ndx\\&+bS^2(t^2-t_*^2)t\int_{\mathbb R^4}\nabla\bar{u}_c\nabla U_ndx-t^3\int_{\mathbb R^4}\bar{u}_cU_n^3dx+(t^2+t^4)\left[O\left(\frac{\log(1+n^2)}{n^2}\right)\right]\\&\leq\frac{a}{2}|\nabla\bar{u}_c|_2^2+\frac{b}{4}|\nabla\bar{u}_c|_2^4-\frac{1}{4}|\bar{u}_c|^4_4-\frac{\mu}{q}|\bar{u}_c|_q^q+S^2\left[\frac{a+b|\nabla \bar{u}_c|_2^2}{2}t^2-\frac{1-bS^2}{4}t^4\right]\\&-\frac{B_0t^3}{n}+bS^2(t^2-t_*^2)t\left[O\left(\frac{\sqrt{\log(1+n^2)}}{n}\right)\right]+(t^2+t^4)\left[O\left(\frac{\log(1+n^2)}{n^2}\right)\right]\\&\leq\frac{a}{2}|\nabla\bar{u}_c|_2^2+\frac{b}{4}|\nabla\bar{u}_c|_2^4-\frac{1}{4}|\bar{u}_c|^4_4-\frac{\mu}{q}|\bar{u}_c|_q^q+\Lambda+\frac{abS^2|\nabla\bar{u}_c|^2_2}{2(1-bS^2)}+\frac{b^2S^2|\nabla\bar{u}_c|^4_2}{4(1-bS^2)}-\frac{C}{n}\\&=J(\bar{u}_c)+\Lambda-\frac{C}{n}=\bar{m}(c)+\Lambda-\frac{C}{n}, \ \ \forall \ t>0,
    \end{aligned}
    \end{equation} which implies that there exists $\bar{n}\in \mathbb N$ such that \begin{equation}\label{sup}
        \sup_{t>0}J(W_{\bar{n},t})<\bar{m}(c)+\Lambda.
    \end{equation}Then, by \eqref{uc+Un}, \eqref{W}, \eqref{nW} and \eqref{Wq}, we have \begin{equation*}
    W_{\bar{n},t}:=\bar{\tau}\left[\bar{u}_c(\bar{\tau} x)+tU_{\bar{n}}(\bar{\tau} x)\right],\ \  |W_{\bar{n},t}|_2^2=c
    \end{equation*}and \begin{equation*}
        \begin{aligned}
            |\nabla W_{\bar{n},t}|_2^2&=|\nabla(\bar{u}_c+tU_{\bar{n}})|_2^2\\&=|\nabla \bar{u}_c|_2^2+t^2|\nabla U_{\bar{n}}|_2^2+2t\int_{\mathbb R^4}\nabla\bar{u}_c\nabla U_{\bar{n}}dx,
        \end{aligned}
    \end{equation*}where $$\bar{\tau}^2=1+\frac{2t}{c}\int_{\mathbb R^4}\bar{u}_c U_{\bar{n}}dx+t^2|U_{\bar{n}}|_2^2,$$which implies that $W_{\bar{n},t}\in S_c$ for all $t>0$, $W_{\bar{n},0}=\bar{u}_c$, and there exists a $\bar{t}>0$  large enough such that $I(W_{\bar{n},\bar{t}})<2m(c)$. Let $\gamma_{\bar{n}}(t):=W_{\bar{n},t\bar{t}}$. Then we can check that $\gamma_{\bar{n}}\in \Gamma$ defined in \eqref{mountainpass}. Thus, it follows from \eqref{M} and \eqref{sup} that \eqref{Mm} holds. The proof is completed.
\end{proof}
\begin{proof}[The proof of Theorem \ref{th7}]
    It follows from Lemma \ref{ps} there exists a sequence $\{u_n\}\subset S_c$ such that \begin{equation}\label{ps_un}
        I(u_n)\to M(c),\ \ I|'_{S_c}(u_n)\to0\ \ {\rm and} \ \ P(u_n)\to0,\ \ {\rm as} \ \ n\to\infty.
    \end{equation} Using \eqref{G-N} and \eqref{ps_un}, one obtains that \begin{equation*}
    \begin{aligned}
        M(c)+o(1)&=I(u_n)-\frac{1}{4}P(u_n)\\&=\frac{a}{4}|\nabla u_n|_2^2-\frac{\mu(4-q)}{2q}|u_n|_q^q\\&\geq\frac{a}{4}|\nabla u_n|_2^2-\frac{\mu(4-q)}{2q}C_q^qc^{\frac{4-q}{2}}|\nabla u_n|_2^{2q-4},
    \end{aligned}
    \end{equation*} which implies that $\{u_n\}$ is bounded in  $ H_{rad}^1(\mathbb R^4)$ due to $2<q<3$. Then, we can assume that there exists a $\hat u\in H_{rad}^1(\mathbb R^4)$ such that, up to a subsequence, \begin{equation}
    \begin{cases}
		u_n\rightharpoonup \hat{u} &{\rm in} \ \ H_{rad}^1(\mathbb R^4);\\
		u_n\to\hat{u} &{\rm in} \ \ L^s(\mathbb R^4), \ \ \forall \ s\in(2,4);\\
        u_n(x)\to\hat{u}(x), &{\rm a.e. \ on} \ \ \mathbb R^4.
	\end{cases}
    \end{equation} By \eqref{ps_un} and the Lagrange multipliers rule, there exists a sequence $\{\lambda_n\}\in \mathbb R$ such that \begin{equation}\label{weak}
    \begin{split}
        a\int_{\mathbb R^4}\nabla u_n \nabla \varphi dx &+b|\nabla u_n|_2^2\int_{\mathbb R^4}\nabla u_n \nabla \varphi dx-\int_{\mathbb R^4}|u_n|^2u_n\varphi dx-\mu\int_{\mathbb R^4}|u_n|^{q-2}u_n \varphi dx\\ &+\int_{\mathbb R^4}\lambda_n u_n\varphi dx=o(1)||\varphi||, \ \ \forall \ \varphi\in H_{rad}^1(\mathbb R^4).
    \end{split}
    \end{equation}In particular, taking $\varphi=u_n$ in \eqref{weak}, then we have \begin{equation}\label{lam_sequence}
        -\lambda_nc=a|\nabla u_n|_2^2+b|\nabla u_n|_2^4-|u_n|_4^4-\mu|u_n|_q^q.
    \end{equation}Since $\{u_n\}$ is bounded in $H_{rad}^1(\mathbb R^4)$, we know that $\{\lambda_n\}$ is bounded in $\mathbb R$. Thus, there exists a $\hat{\lambda}_c\in \mathbb R$ such that, up to a subsequence, $\lambda_n\to \hat{\lambda}_c$ as $n\to \infty$.

    Now,  we claim that $\hat{u}\neq0$. Indeed, if $\hat{u}=0$, then we have $u_n\to0$ in $L^q(\mathbb R^4)$. Then, it follows from \eqref{G-N} and \eqref{ps_un} that \begin{equation}
    \begin{split}
        o(1)&=a|\nabla u_n|_2^2+b|\nabla u_n|_2^4-|u_n|_4^4 \geq a|\nabla u_n|_2^2+(b-S^{-2})|\nabla u_n|_2^4.
    \end{split}
    \end{equation} Set that $A:=\lim_{n\to \infty}|\nabla u_n|_2^2\geq0$. Then we have  $aA+(b-S^{-2})A^2\leq0$, which indicates that either $A\geq\frac{aS^2}{1-bS^2}$ or $A=0$. If $A\geq\frac{aS^2}{1-bS^2}$, combining this with \eqref{ps_un}, we have \begin{equation*}
    \begin{aligned}
            M(c)+o(1)&=I(u_n)-\frac{1}{4}P(u_n)=\frac{a}{4}|\nabla u_n|_2^2\geq \frac{aS^2}{4(1-bS^2)}=\Lambda,
    \end{aligned}
    \end{equation*} which contradicts with \eqref{Mm}. If $A=0$, it follows from that \eqref{ps_un} that $M(c)=0$, which is impossible. Therefore, $\hat u\neq0$. Moreover, in view of \eqref{ps_un}, \eqref{lam_sequence} and $2<q<3$, we have \begin{equation*}
    \begin{aligned}
    \hat{\lambda}_cc=\lim_{n\to\infty}\lambda_nc&=\lim_{n\to\infty}-(a|\nabla u_n|_2^2+b|\nabla u_n|_2^4-|u_n|_4^4-\mu|u_n|_q^q)\\
    &=\lim_{n\to\infty}\mu(\frac{4}{q}-1)|u_n|_q^q=\mu(\frac{4}{q}-1)|\hat u|_q^q>0,
    \end{aligned}
    \end{equation*}which implies that $\hat{\lambda}_c>0$.

    Letting $n\to \infty$ in \eqref{weak}, we have \begin{equation}\label{Alam_sequence}
         a\int_{\mathbb R^4}\nabla \hat u\nabla \varphi dx +bA\int_{\mathbb R^4}\nabla \hat u\nabla \varphi dx+\int_{\mathbb R^4}\hat{\lambda}_c\hat u\varphi dx=\int_{\mathbb R^4}|\hat u|^2\hat u\varphi dx+\mu\int_{\mathbb R^4}|\hat u|^{q-2}\hat u \varphi dx,
    \end{equation} which means that $\hat u$ is a weak solution of the following equation
    \begin{equation*}
        -(a+bA)\Delta u+\hat{\lambda}_c u=\mu|u|^{q-2}u+|u|^3u \ \ {\rm in} \ \ \mathbb R^4.
    \end{equation*}
    Then, $\hat{u}$ satisfies the following Pohozaev identity\begin{equation}\label{Apohozaev}
    (a+bA)|\nabla\hat{u}|_2^2+2\hat{\lambda}_c|\hat{u}|_2^2-\frac{4\mu}{q}|\hat{u}|_q^q-|\hat{u}|_4^4.
    \end{equation}Let $\varphi=\hat{u} $ in \eqref{Alam_sequence}, together with \eqref{Apohozaev}, we can deduce that \begin{equation}\label{hP}
        \hat{P}(\hat u):=(a+bA)|\nabla\hat u|_2^2-|\hat u|_4^4-\frac{2\mu(q-2)}{q}|\hat u|_q^q=0.
    \end{equation}
    Set that $v_n:=u_n-\hat u$. By the Brezis-Lieb Lemma, we have \begin{equation}\label{BL}
        A=|\nabla u_n|_2^2+o(1)=|\nabla v_n|_2^2+|\nabla\hat u|_2^2+o(1) \ \ {\rm and} \ \ |u_n|_4^4=|v_n|_4^4+|\hat u|_4^4+o(1).
    \end{equation}
    It follows from \eqref{ps_un}, \eqref{hP} and \eqref{BL} that \begin{equation}\label{RP}
    \begin{aligned}
          o(1)=P(u_n)&=a|\nabla u_n|_2^2+b|\nabla u_n|_4^4-|u_n|^4_4-\frac{2\mu(q-2)}{q}|u_n|^q_q\\
          &=\hat{P}(\hat u)+(a+bA)|\nabla v_n|_2^2-|v_n|_4^4+o(1)=(a+bA)|\nabla v_n|_2^2-|v_n|_4^4+o(1)\\
          &=(a+b|\nabla \hat u|_2^2)|\nabla v_n|_2^2+b|\nabla v_n|_2^4-|v_n|_4^4+o(1),
    \end{aligned}
    \end{equation}
    together with \eqref{ps_un} and \eqref{BL}, one obtains that \begin{equation}\label{RM}
        \begin{aligned}
            M(c)+o(1)&=\frac{a}{2}|\nabla u_n|_2^2+\frac{b}{4}|\nabla u_n|_2^4-\frac{1}{4}|u_n|_4^4-\frac{\mu}{q}|u_n|_q^q\\&=\frac{a}{2}|\nabla v_n|_2^2+\frac{b}{4}|\nabla v_n|_2^4-\frac{1}{4}|v_n|_4^4+\frac{b}{2}|\nabla\hat{u}|_2^2|\nabla v_n|_2^2+I(\hat{u})+o(1)\\& =\frac{a}{4}|\nabla v_n|_2^2+\frac{b}{4}|\nabla\hat{u}|_2^2|\nabla v_n|_2^2+I(\hat{u})+o(1).       \end{aligned}
    \end{equation}
    Set that $l:=\lim_{n\to \infty}|\nabla v_n|_2^2\geq0$. Then it follow from \eqref{G-N} and \eqref{RP} that \begin{equation*}
    \begin{aligned}
        (a+b|\nabla \hat u|_2^2)l+bl^2&=\lim_{n\to\infty}(a+b|\nabla \hat u|_2^2)|\nabla v_n|_2^2+b|\nabla v_n|_2^4=\lim_{n\to\infty}|v_n|_4^4\\
        &\leq \lim_{n\to\infty}S^{-2}|\nabla v_n|_2^4=S^{-2}l^2,
    \end{aligned}
    \end{equation*}
    which indicates that either $l\geq\frac{(a+b|\nabla \hat u|_2^2)S^2}{1-bS^2}$ or $l=0$. If $l\geq\frac{(a+b|\nabla \hat u|_2^2)S^2}{1-bS^2}$, we consider the two cases: $|\nabla \hat{u}|_2^2<k_0$ and $|\nabla \hat{u}|_2^2\geq k_0$.

    $(i)$ If $|\nabla \hat{u}|_2^2<k_0$, by Lemmas \ref{lem3.3} and \ref{lem3.5}, we have $$J(\hat u)\geq \bar{m}(|\hat u|_2^2)\geq \bar{m}(c).$$ Combining this with \eqref{RM}, one obtains that \begin{equation*}
        \begin{aligned}
           M(c)+o(1)&=\frac{a}{4}|\nabla v_n|_2^2+\frac{b}{4}|\nabla\hat{u}|_2^2|\nabla v_n|_2^2+I(\hat{u})+o(1)\\&=\frac{a+b|\nabla\hat{u}|_2^2}{4}l+I(\hat{u})+o(1)\\&\geq \frac{(a+b|\nabla\hat{u}|_2^2)^2S^2}{4(1-bS^2)}+I(\hat{u})+o(1)\\&=\frac{a^2S^2}{4(1-bS^2)}+\frac{abS^2|\nabla \hat u|_2^2}{2(1-bS^2)}+\frac{b^2S^2|\nabla \hat u|_2^4}{4(1-bS^2)}+I(\hat{u})+o(1)\\&=\Lambda+J(\hat u)+o(1)\\&\geq \Lambda+\bar{m}(c)+o(1),
        \end{aligned}
    \end{equation*} which contradicts with \eqref{Mm}.

    $(ii)$ If $|\nabla \hat{u}|_2^2\geq k_0$, by \eqref{G-N},  \eqref{c1}, \eqref{hP}, \eqref{BL} and \eqref{RM}, one obtains that \begin{equation*}
        \begin{aligned}
             M(c)+o(1)&=\frac{a}{4}|\nabla v_n|_2^2+\frac{b}{4}|\nabla\hat{u}|_2^2|\nabla v_n|_2^2+I(\hat{u})-\frac{1}{4}\hat{P}(\hat{u})+o(1)\\&=\frac{a}{4}|\nabla\hat{u}|_2^2+\frac{a}{4}|\nabla v_n|_2^2-\frac{\mu(4-q)}{2q}|\hat u|_q^q+o(1)\\&=\frac{a}{4}|\nabla\hat{u}|_2^2+\frac{a}{4}l-\frac{\mu(4-q)}{2q}|\hat u|_q^q+o(1)\\&\geq \frac{a^2S^2}{4(1-bS^2)}+\frac{a}{4(1-bS^2)}|\nabla\hat{u}|_2^2-\frac{\mu(4-q)}{2q}|\hat u|_q^q+o(1)\\&\geq \frac{a^2S^2}{4(1-bS^2)}+\frac{a}{4(1-bS^2)}|\nabla\hat{u}|_2^2-\frac{\mu(4-q)}{2q}C_q^qc^{\frac{4-q}{2}}|\nabla\hat{u}|_2^{2q-4}+o(1)\\&= \frac{a^2S^2}{4(1-bS^2)}+\left[\frac{a}{4(1-bS^2)}|\nabla\hat{u}|_2^{6-2q}-\frac{\mu(4-q)}{2q}C_q^qc^{\frac{4-q}{2}}\right]|\nabla\hat{u}|_2^{2q-4}+o(1)\\&\geq\frac{a^2S^2}{4(1-bS^2)}+\left[\frac{a}{4(1-bS^2)}k_0^{3-q}-\frac{\mu(4-q)}{2q}C_q^qc^{\frac{4-q}{2}}\right]|\nabla\hat{u}|_2^{2q-4}+o(1)\\&\geq \Lambda+o(1),
        \end{aligned}
    \end{equation*} which also contradicts with \eqref{Mm}. Therefore, $l=0$ holds. Then we know that $|\nabla u_n|_2\to|\nabla\hat u|_2$ and $|u_n|_4\to|\hat u|_4$ as $n\to \infty$. Moreover, letting $\varphi=u_n-\hat{u}$ in \eqref{weak} and \eqref{Alam_sequence} respectively and subtracting them, we have $$(a+bA)|\nabla(u_n-\hat u)|_2^2-\lambda_c|u_n-\hat u|_2^2\to 0 \  {\rm as} \ \ n\to\infty,$$which implies that $|u_n|_2\to|\hat u|_2$ as $n\to\infty$. Thus, $u_n\to\hat u$ in $H^1_{rad}(\mathbb R^4)$. The proof is complete.
\end{proof}
\section*{}

\noindent {\bf Author's contributions:} All authors contributed equally to this work. Qilin Xie  and Jianshe Yu are the corresponding authors.
\smallskip

\bibliography{references}

\begin{thebibliography}{10}

\bibitem{al2014bending}
M.~Al-Gwaiz, V.~Benci, and F.~Gazzola.
\newblock Bending and stretching energies in a rectangular plate modeling
  suspension bridges.
\newblock {\em Nonlinear Anal.}, 106:18--34, 2014.

\bibitem{bartsch2019multiple}
T.~Bartsch and N.~Soave.
\newblock Multiple normalized solutions for a competing system of
  {Schr{\"o}dinger} equations.
\newblock {\em Calc. Var. Partial Differ. Equ.}, 58:1--24, 2019.

\bibitem{brezis1979remarks}
H.~Br{\'e}zis and T.~Kato.
\newblock Remarks on the {Schr{\"o}dinger} operator with singular complex
  potentials.
\newblock {\em J. Math. Pures Appl.}, 58(2):137--151, 1979.

\bibitem{carriao2022normalized}
P.~C. Carri{\~a}o, O.~H. Miyagaki, and A.~Vicente.
\newblock Normalized solutions of {Kirchhoff} equations with critical and
  subcritical nonlinearities: the defocusing case.
\newblock {\em Partial Differ. Equ. Appl.}, 3(5):64, 2022.

\bibitem{chen2011nehari}
C.~Chen, Y.~Kuo, and T.~Wu.
\newblock The {Nehari} manifold for a {Kirchhoff} type problem involving
  sign-changing weight functions.
\newblock {\em J. Differential Equations}, 250(4):1876--1908, 2011.

\bibitem{chen2024another}
S.~Chen and X.~Tang.
\newblock Another look at {Schr{\"o}dinger} equations with prescribed mass.
\newblock {\em J. Differential Equations}, 386:435--479, 2024.

\bibitem{chen2024normalized}
S.~Chen and X.~Tang.
\newblock Normalized solutions for {Kirchhoff} equations with {Sobolev}
  critical exponent and mixed nonlinearities.
\newblock {\em Math. Ann.}, 391(2):2783--2836, 2025.

\bibitem{chipot1997some}
M.~Chipot and B.~Lovat.
\newblock Some remarks on nonlocal elliptic and parabolic problems.
\newblock {\em Nonlinear Anal.}, 30(7):4619--4627, 1997.

\bibitem{correa2005class}
F.~J.~S. Corr{\^e}a and D.~C. de~Morais~Filho.
\newblock On a class of nonlocal elliptic problems via {Galerkin} method.
\newblock {\em J. Math. Anal. Appl.}, 310(1):177--187, 2005.

\bibitem{fang2024normalized}
X.~Fang, Z.~Ou, and Y.~Lv.
\newblock Normalized solutions to the {Kirchhoff} {Equation} with triple
  critical exponents in $\mathbb {R}^{4}$.
\newblock {\em Appl. Math. Lett.}, 153:109062, 2024.

\bibitem{feng2023normalized}
X.~Feng, H.~Liu, and Z.~Zhang.
\newblock Normalized solutions for {Kirchhoff} type equations with combined
  nonlinearities: The {Sobolev} critical case.
\newblock {\em Discrete Contin. Dyn. Syst.}, 43(8), 2023.

\bibitem{figueiredo2014existence}
G.~M. Figueiredo, N.~Ikoma, and J.~R. Santos~J{\'u}nior.
\newblock Existence and concentration result for the {Kirchhoff} type equations
  with general nonlinearities.
\newblock {\em Arch. Rational Mech. Anal.}, 213:931--979, 2014.

\bibitem{guo2018blow}
H.~Guo, Y.~Zhang, and H.~Zhou.
\newblock Blow-up solutions for a {Kirchhoff} type elliptic equation with
  trapping potential.
\newblock {\em Commun. Pure Appl. Anal.}, 17(5):1875--1897, 2018.

\bibitem{he2012existence}
X.~He and W.~Zou.
\newblock Existence and concentration behavior of positive solutions for a
  {Kirchhoff} equation in $\mathbb {R}^{3}$.
\newblock {\em J. Differential Equations}, 252(2):1813--1834, 2012.

\bibitem{hu2023normalized}
J.~Hu and A.~Mao.
\newblock Normalized solutions to the {Kirchhoff} equation with a perturbation
  term.
\newblock {\em Differ. Integr. Equ.}, 36(3/4):289--312, 2023.

\bibitem{jeanjean2022orbital}
L.~Jeanjean, J.~Jendrej, T.~T. Le, and N.~Visciglia.
\newblock Orbital stability of ground states for a {Sobolev} critical
  {Schr{\"o}dinger} equation.
\newblock {\em J. Math. Pures Appl.}, 164:158--179, 2022.

\bibitem{jeanjean2022multiple}
L.~Jeanjean and T.~T. Le.
\newblock Multiple normalized solutions for a {Sobolev} critical
  {Schr{\"o}dinger} equation.
\newblock {\em Math. Ann.}, 384(1):101--134, 2022.

\bibitem{kirchhoffvorlesungen}
G.~Kirchhoff.
\newblock Mechanik, {Teubner}, {Leipzig}, 1883.

\bibitem{kong2023normalized}
L.~Kong and H.~Chen.
\newblock Normalized ground states for the mass-energy doubly critical
  {Kirchhoff} equations.
\newblock {\em Acta Appl. Math.}, 186(1):5, 2023.

\bibitem{li2022normalized}
G.~Li, X.~Luo, and T.~Yang.
\newblock Normalized solutions to a class of {Kirchhoff} equations with
  {Sobolev} critical exponent.
\newblock {\em Ann. Fenn. Math}, 47(2):895--925, 2022.

\bibitem{li2019existence}
Y.~Li, X.~Hao, and J.~Shi.
\newblock The existence of constrained minimizers for a class of nonlinear
  {Kirchhoff}--{Schr{\"o}dinger} equations with doubly critical exponents in
  dimension four.
\newblock {\em Nonlinear Anal.}, 186:99--112, 2019.

\bibitem{lieb2001analysis}
E.~Lieb and M.~Loss.
\newblock Analysis, volume 14 of {Graduate} {Studies} in {Mathematics}, vol. 4.
\newblock {\em A M S}, 2001.

\bibitem{lions1978some}
J.-L. Lions.
\newblock On some questions in boundary value problems of mathematical physics.
\newblock In {\em North-Holland Math. Stud.}, volume~30, pages 284--346.
  Elsevier, 1978.

\bibitem{weinberger1967maximum}
M.~H. Protter and H.~F. Weinberger.
\newblock {\em Maximum principles in differential equations}.
\newblock Prentice-Hall, Inc., Englewood Cliffs, N.J., 1967.

\bibitem{Qizou2022}
S.~Qi and W.~Zou.
\newblock Exact number of positive solutions for the {Kirchhoff} equation.
\newblock {\em SIAM J. Math. Anal.}, 54(5):5424--5446, 2022.

\bibitem{soave2020normalized}
N.~Soave.
\newblock Normalized ground states for the {NLS} equation with combined
  nonlinearities.
\newblock {\em J. Differential Equations}, 269(9):6941--6987, 2020.

\bibitem{soave2020sobolev}
N.~Soave.
\newblock Normalized ground states for the {NLS} equation with combined
  nonlinearities: the {Sobolev} critical case.
\newblock {\em J. Funct. Anal.}, 279(6):108610, 2020.

\bibitem{wei2022normalized}
J.~Wei and Y.~Wu.
\newblock Normalized solutions for {Schr{\"o}dinger} equations with critical
  {Sobolev} exponent and mixed nonlinearities.
\newblock {\em J. Funct. Anal.}, 283(6):109574, 2022.

\bibitem{minimax}
M.~Willem.
\newblock {\em Minimax theorems}, volume~24 of {\em Progress in Nonlinear
  {Differential} {Equations} and their {Applications}}.
\newblock Birkh\"auser Boston, Inc., Boston, MA, 1996.

\bibitem{xie2016bound}
Q.~Xie, S.~Ma, and X.~Zhang.
\newblock Bound state solutions of {K}irchhoff type problems with critical
  exponent.
\newblock {\em J. Differential Equations}, 261(2):890--924, 2016.

\bibitem{Xie2023existence}
Q.~Xie, X.~Wu, and C.~Tang.
\newblock Existence and multiplicity of solutions for {K}irchhoff type problem
  with critical exponent.
\newblock {\em Commun. Pure Appl. Anal.}, 12(6):2773--2786, 2013.

\bibitem{xie2022study}
Q.~Xie and B.-X. Zhou.
\newblock A study on the critical {Kirchhoff} problem in high-dimensional
  space.
\newblock {\em Z. Angew. Math. Phys.}, 73(1):4, 2022.

\bibitem{xu2024existence}
L.~Xu, F.~Li, and Q.~Xie.
\newblock Existence and multiplicity of normalized solutions with positive
  energy for the {K}irchhoff equation.
\newblock {\em Qual. Theory Dyn. Syst.}, 23(3):Paper No. 135, 23, 2024.

\bibitem{xu2024multiplicity}
L.~Xu, C.~Song, and Q.~Xie.
\newblock Multiplicity of normalized solutions for {S}chr\"odinger equation
  with mixed nonlinearity.
\newblock {\em Taiwanese J. Math.}, 28(3):589--609, 2024.

\bibitem{xulin2024}
L.~Xu and Q.~Xie.
\newblock The complementarity of normalized solutions for {K}irchhoff equations
  with mixed nonlinearity.
\newblock {\em J. Math. Res. Appl.}, 44(4):483--494, 2024.

\bibitem{ye2015existence}
H.~Ye.
\newblock The existence of normalized solutions for {$L^{2}$-critical}
  constrained problems related to {Kirchhoff} equations.
\newblock {\em Z. Angew. Math. Phys.}, 66:1483--1497, 2015.

\bibitem{ye2015sharp}
H.~Ye.
\newblock The sharp existence of constrained minimizers for a class of
  nonlinear {Kirchhoff} equations.
\newblock {\em Math. Methods Appl. Sci.}, 38(13):2663--2679, 2015.

\bibitem{ye2016mass}
H.~Ye.
\newblock The mass concentration phenomenon for {$L^{2}$-critical} constrained
  problems related to {Kirchhoff} equations.
\newblock {\em Z. Angew. Math. Phys.}, 67:1--16, 2016.

\bibitem{zeng2017existence}
X.~Zeng and Y.~Zhang.
\newblock Existence and uniqueness of normalized solutions for the {Kirchhoff}
  equation.
\newblock {\em Appl. Math. Lett.}, 74:52--59, 2017.

\bibitem{zhang2022normalized}
J.~Zhang, J.~Zhang, and X.~Zhong.
\newblock Normalized solutions to {Kirchhoff} type equations with a critical
  growth nonlinearity.
\newblock {\em arXiv preprint arXiv:2210.12911}, 2022.

\end{thebibliography}
\bibliographystyle{abbrv}
\end{document}